\documentclass[12pt]{amsart}
\usepackage[left=3cm, right=3cm, top=2.6cm, bottom=2.5 
cm]{geometry}

\usepackage[utf8]{inputenc}
\usepackage[english, vietnamese]{babel}

\usepackage{amssymb,amsfonts,amsthm,amsmath,enumerate,amscd,
tikz}
\usepackage{colordvi,color}

\usepackage{hyperref}
\usepackage{setspace}

\newcommand\invisible[1]{}

\makeatletter
\def\subsection{\@startsection{subsection}{2}%
	\z@{.5\linespacing\@plus.7\linespacing}{.1\linespacing}%
	{\normalfont\bfseries}}
\makeatother
\newtheorem{theorem}{Theorem}[section]
\newtheorem{lemma}[theorem]{Lemma}
\newtheorem{corollary}[theorem]{Corollary}

\newtheorem{proposition}[theorem]{Proposition}
\newtheorem{definition}[theorem]{Definition}
\newtheorem{example}[theorem]{Example}
\newtheorem{remark}[theorem]{Remark}
\newtheorem{propdef}[theorem]{Proposition-Definition}
\newcommand{\field}[1]{\mathbb{#1}}
\newcommand{\N}{\field{N}}
\newcommand{\R}{\field{R}}
\newcommand{\bun}{{\bf 1}}
\newcommand{\ba}{{\bf a}}
\newcommand{\bbl}{{\bf bl}}
\newcommand{\bbo}{{\bf 0}}
\newcommand{\bB}{{\bf B}}

\newcommand{\bc}{{\bf c}}
\newcommand{\bS}{{\bf S}}
\newcommand{\bu}{{\bf u}}
\newcommand{\bv}{{\bf v}}
\newcommand{\bw}{{\bf w}}
\newcommand{\bx}{{\bf x}}
\newcommand{\by}{{\bf y}}
\newcommand{\bz}{{\bf z}}
\newcommand{\cC}{\mathcal{C}}
\newcommand{\cS}{\mathcal{S}}
\newcommand{\cU}{\mathcal{U}}
\newcommand{\cX}{\mathcal{X}}
\newcommand{\cyl}{{\rm cyl}}
\newcommand{\Cyl}{{\rm Cyl}}
\newcommand{\clos}{{\rm clos}}
\newcommand{\dbar}{\overline{d}}
\newcommand{\dd}{{\partial}}

\newcommand{\dlt}{{\delta}}
\newcommand{\Ebar}{\overline{E}}
\newcommand{\etabar}{\overline{\eta}}

\newcommand{\gbar}{\overline{g}}

\newcommand{\inn}{{{\rm in}}}

\newcommand{\lbd}{{\lambda}}

\newcommand{\mdl}{{\rm mod}}
\newcommand{\Mbar}{\overline{M}}
\newcommand{\Nbar}{\overline{N}}
\newcommand{\omg}{\omega}
\newcommand{\ovP}{{\overline{P}}}
\newcommand{\ovX}{{\overline{X}}}
\newcommand{\rd}{{\rm d}}
\newcommand{\Rgo}{{\R_{\geq 0}}}
\newcommand{\Rn}{{\R^n}}
\newcommand{\sing}{{\rm sing}}

\newcommand{\Sgm}{\Sigma}
\newcommand{\sm}{{\rm sm}}
\newcommand{\tht}{{\theta}}
\newcommand{\tr}{\tilde{r}}
\newcommand{\Ubar}{\overline{U}}
\newcommand{\ve}{\varepsilon}
\newcommand{\vp}{\varphi}
\newcommand{\vpbar}{\overline{\varphi}}
\newcommand{\stt}{\textup{st}}

\numberwithin{equation}{section}

\begin{document}
\title[On Lipschitz geometry of globally conic singular 
manifolds]{Remarks on Lipschitz geometry\\ on globally conic 
singular manifolds}
\author[A. Costa]{Andr\'e Costa}
%
\author[V. Grandjean]{Vincent Grandjean}
\author[M. Michalska]{Maria Michalska}
\address{
A. Costa, Centro de Ci\^encias e tecnologias, 
Universidade Estadual do Cear\'a, 
Campus do Itaperi, 60.714-903 Fortaleza - CE, Brasil}
\address{V. Grandjean, 
Departamento de Matem\'atica, Departamento de Matem\'atica,
Universidade Federal de Santa Catarina, 
88.040-900 Florianópolis - SC, Brasil,
Brasil}
\address{M. Michalska, Wydzia\l{} Matematyki i Informatyki, 
Uniwersytet \L{}\'o{}dzki, Banacha 22, 90-238 \L{}\'o{}d\'z{}, 
Poland}
\email{andrecosta.math@gmail.com, 
vincent.grandjean@ufsc.br,\newline 
maria.michalska@wmii.uni.lodz.pl}
\subjclass[2000]{}
\dedicatory{In honour of Professor 	 Phạm Tiến Sơn   for his 
sixtieth birthday.}
%
%
%
\maketitle
\selectlanguage{english} 
\begin{abstract}
We study metric properties of manifolds with conic singularities and present a natural interplay between metrically conic and metrically asymptotically conic behaviour. 
As a consequence, we prove that a singular sub-manifold is Lipschitz normally 
embedded, i.e. its inner and outer metric structures are equivalent, in an ambient singular manifold, whenever the 
singularities are conic and the ends of the manifold are 
asymptotically conic, which answers positively a question 
of \cite{CoGrMi3}.
\end{abstract}
\setcounter{tocdepth}{1}
{\setstretch{0.8}
\tableofcontents
}

\section*{Introduction}

The aim of this note is to explain within the context of conic
 singular manifolds
how the results on Lipschitz geometry of  algebraic sets of
our recent works \cite{CoGrMi1,CoGrMi3}, compare \cite{CoGrMi2},  
are occurrences of a general simple principle.

A diffeomorphism between two smooth Riemannian manifolds is always a locally
bi-Lipschitz mapping. Yet in
presence of singularities a counterpart statement is not clear.  Thus we investigate this question here for metrically
conic singularities. We will showcase the relations between  conic singular
manifolds, manifolds with  conic metric and their quotient~spaces in Section~\ref{section:GCSM}.

The natural interplay between metrically conic and asymptotically metrically
conic behaviour is presented in Corollary \ref{cor:conic-complet}  
and Proposition \ref{prop:acs-submanif-complet}. Thus  
Lipschitz geometry of globally conic singular manifolds becomes much clearer and results conjectured in \cite{CoGrMi3}
follow in Theorem~\ref{thm:MainAsympConic} which states that
\em every connected globally conic singular sub-manifold of a globally conic
singular manifold is Lipschitz normally embedded, \em i.e. the inner and outer metric structures are equivalent. 

%
%
%
%
%
%
%
%
%
%
%
%
%
%
%
%
%
%
%
%
%
%
%
%
%
%
%
%
%
%
%
%
%
%
%
%
%
%
%
%
%
%
%
%
%
%
%
%
\section{Preliminaries}
Throughout the paper smooth means $C^\infty$-smooth, all manifolds without or with 
boundary are smooth, all Riemannian 
metrics  are assumed smooth.
\subsection{Metric structures of subsets}\label{secPseudoDist}

For convenience we allow a distance function to attain value $\infty$, for instance for points in different
connected components of a set. 
Thus isometries of metric spaces are  connected-component-wise isometries.

A pair $(M,d)$ is a \em pseudo-metric space, \em if the function 
$d:M\times M \to \Rgo\cup\{\infty\}$ is a \em pseudo-distance, \em
i.e. for all $\bx,\bx',\bx''\in M$ it satisfies the following conditions: 
$$
d(\bx,\bx) = 0, \;\; d(\bx,\bx')=d(\bx',\bx)\;\; {\rm and} \;\; 
d(\bx,\bx'')\leq d(\bx,\bx')+ d(\bx',\bx'').
$$
Let $(M,d)$ be a pseudo-metric space. The \em length 
 of an arc $\bc:[a,b]\to M$ \em 
is 
$$
\ell(\bc) = \sup \left|\sum_{i=1}^k d(\bc(t_{i-1}), \bc(t_{i})) \right|,
$$
where the supremum is taken over all partitions $a=t_0<\dots<t_k=b, k\in \N,$ of the interval $I$.
\begin{definition}\label{def:in-out-dist}
Let $(M,d)$ be a pseudo-metric space and $X$ be a subset of $M$. 
	
(i) The \em outer pseudo-metric structure is $(X,d_X)$, \em where $d_X$ is 
the restriction of the pseudo-distance in~$M$ to~$X$.
	
(ii) The \em inner pseudo-metric structure is $(X,d_\inn^X)$, \em where $d_\inn^X$
is the infimum of lengths of arcs within $X$ 
connecting a given pair of points. 

\end{definition}

Whenever the inner (rep. outer) pseudo-distance is a distance on $X$, the inner (resp. outer) structure is naturally an inner (resp. outer) metric structure on $X$.

A mapping $f:(M,d)\to (M',d')$ between pseudo-metric spaces is \em pseudo-Lipschitz \em 
if there exists a positive constant $L$ such that the inequality $d'(f(x),f(y))\leq L d (x,y)$ is satisfied for all $x,y\in M$.

A special type of subsets of a (pseudo-)metric space are those for which their 
outer and inner (pseudo-)metric structure are indistinguishable from the 
(pseudo-)Lipschitz point of view:

\begin{definition}\label{defLNEset}
Let $(M,d)$ be a pseudo-metric space and $X$ be a subset of $M$ such that its inner
and outer pseudo-distances in $X$ are both distance functions. 
	
(i) The subset $X$  is \em Lipschitz normally embedded \em 
(shortened to LNE) in $M$ if its  outer and inner distances are equivalent,
i.e. the identity map $$ \textup{id}: (X, d_X) \to (X,d_\inn^X) $$
is bi-Lipschitz.

(ii) The subset $X$  is \em locally LNE at $\bx$ \em if there 
exists a neighbourhood $U$ of $\bx$ in~$M$ such that $X\cap U$ is LNE in $M$.
	
(iii) The subset $X$   is \em locally LNE \em if it is locally 
LNE at {each point of its closure in $M$}.
	
(iv) The space $(X,d_X)$ is a \em length space \em when 
the identity map of point (i) is an isometry.
\end{definition}
When writing that a subset of a metric space is LNE, it is always understood that
it is LNE with respect to the metric structure  of the given ambient space. By
convention, the empty set is LNE. This notion is the same as Whitney's $P$ condition of 
\cite{Whi1} and Gromov's quasi-convexity of \cite{Gro1}. We use the least ambivalent name of 
Lipschitz normally embedded after the paper~\cite{BiMo}.

\subsection{Singular metrics and pseudo-distances on manifolds}\label{section:pseudo-metric}
Let $(M,\dd M)$ be a smooth manifold with boundary (possibly empty). A \em singular Riemannian 
metric over a manifold $M$ \em is a smooth $2$-symmetric tensor over $M$ everywhere 
positive semi-definite. The \em singular locus $\sing(g)$ \em of the singular 
Riemannian metric $g$ over $M$ is the set of points of $M$ at which $g$ is not 
positive-definite. It is a closed subset of $M$ and $g_{|M\setminus \sing(g)}$ is a Riemannian metric.

Two singular Riemannian metrics $g$ and $h$ over $(M,\dd M)$ are \em equivalent
\em if there exist constants $a,b>0$ such that 
$$
a\,g(\bx)(\xi,\xi) \, \leq \, h(\bx)(\xi,\xi) \, \leq b \,g(\bx)(\xi,\xi), \;\; \forall (\bx,\xi)\in TM.
$$
Given a singular Riemannian metric $g$ over a manifold $M$, the length of a smooth curve $\bc : [a,b] \to M$ is
$$
\ell_g(\bc) = \int_a^b \big(g(\bc'(t),\bc'(t))\big)^\frac{1}{2}\,\rd t.
$$
Then $M$ is equipped with the pseudo-distance $d_g$ obtained by taking the infimum of the 
length of piecewise smooth curves connecting a given pair of points, this pseudo-distance 
is a distance function on $M\setminus \sing(g)$.

\begin{remark}
By for instance \cite{Burago,Pal}, given a length structure $d$ on the manifold $M$, smooth
outside of the diagonal $M\times M$, there exists a Riemannian
metric $g$ on $M$ such that $d=d_g$. By Myers-Steenrod theorem a mapping between Riemannian
manifolds is a smooth Riemannian isometry if and only if it is an
isometry of the underlying metric spaces.
\end{remark}

In particular, a diffeomorphism between smooth manifolds is locally bi-Lipschitz regardless
of the choice of Riemannian metrics. Thus the
Lipschitz structure on a manifold depends only on its smooth structure and is independent on the Riemannian structure (this is
already true for $\cC^1$ regularity). 
\subsection{Collar neighbourhoods and p-sub-manifolds}

Let $(M,\dd M)$ be a smooth manifold with boundary $\dd M$ (possibly empty). 

A \em boundary defining function  in $M$ \em is a
smooth function over $M$, positive in $M\setminus\dd M$, vanishing on $\dd M$ 
and with no critical point in $\dd M$. Such smooth functions always exist. 
When $\dd M$ is compact, given any boundary defining function $x$, there 
exists $\eta >0$ such that $[0,\eta]$ contains no critical value of $x$. 
 
An embedded smooth sub-manifold of $(M,\dd M)$ is a subset $N$ such that
(i) it is a smooth manifold $(N,\dd N)$ with boundary $\dd N$ (possibly 
empty); (ii) its manifold topology coincides with the induced topology from 
$M$; (iii) the inclusion mapping $N \hookrightarrow M$ is a smooth injective 
immersion.

\begin{definition}\label{defpsub}(\cite[Section I.7]{Mel})
	Let $(M,\dd M)$ be a smooth manifold with boundary. A subset $N$ of $M$ is a 
	\em p-sub-manifold \em if: (i) it is a smooth sub-manifold with boundary $\dd 
	N$ of $M$; (ii) $\dd N$ is contained in $\dd M$; (iii) $N$ is transverse to 
	$\dd M$:
	$$
	\bx \in \dd N \; \Longrightarrow \; T_\bx N + T_\bx \dd M = T_\bx M.
	$$
\end{definition}
The smooth cylinder over a manifold $N$ of height $\eta$ is 
$$\Cyl(N,\eta) := N \times 
[0,\eta)$$ 
and the open cylinder over $N$ is
$$\Cyl(N,\eta)^o := \Cyl(N,\eta)\setminus (N\times \{0\}) = N \times 
(0,\eta)$$
We recall the collar neighbourhood result below, proof can be found for instance in \cite{CoGrMi3}.
\begin{propdef}\label{prop:collar-ngbhd} 
Let $(M,\dd M)$ be a manifold with smooth compact boundary.  
Let $x$ be a boundary defining function. There exist $\eta >0$ such 
that the function $x$ has no critical value in $[0,\eta]$, an open 
neighbourhood 
$U_x^\eta$ of $\dd M$ in $M$ and a smooth diffeomorphism
\begin{equation}\label{eq:collar-nghbhd}
\phi_x^\eta : U_x^\eta := M\cap\{x<\eta\} \to \Cyl(\dd M,\eta), \;\; 
\bz \mapsto (\by(\bz),x(\bz)),
\end{equation}
which we call the \em collar neighbourhood diffeomorphism. \em 
\end{propdef}

\subsection{Spherical blowing-up}\label{section:blup}
Let $(M,g)$ be a Riemannian manifold possibly with non-empty boundary $\dd M$.
Let $\ba$ be a point of $M\setminus \dd M$. There exists a positive radius $\ve$
such that the exponential map
$$
\exp_\ba: B(\ba,\ve) \to B_g(\ba,\ve)
$$
centred at $\ba$ maps diffeomorphically the open $\ve$-ball $B(\ba,\ve)$ of 
$T_\ba M$ onto the open $\ve$-geodesic ball $B_g(\ba,\ve)$ of $M$. 
Let $\bS_\ba$ be the unit-sphere of $T_\ba M$. The following
mapping 
$$
\rho_\ba : \bS_\ba \times [0,1) \to B(\ba,\ve), \;\; 
(\bu,t) \to \ve\,t\,\bu. 
$$ 
is smooth.

The blowing-up of the Riemannian manifold $(M,g)$ centred at $\ba$ is
the smooth surjective mapping 
\begin{equation}\label{eq:blow-up}
\bbl_\ba^M : [M,\ba] \to M
\end{equation}
where $[M,\ba]$ is a smooth manifold with boundary $\bS_\ba\times 0\sqcup\dd M$. It is obtained as the gluing of $\bS_\ba\times [0,1)$ with $M\setminus \ba$ 
identifying $\bS_\ba\times (0,1)$ with the pointed open ball 
$\bB_g(\ba,\ve)\setminus\ba$ via $\exp_\ba \circ \rho_\ba$.
The mapping $\bbl_\ba$ is just the quotient mapping of the gluing.

Therefore, the spherical blow up $\bbl_\ba$  maps $\bS_\ba\times 0$ onto $\ba$
and induces a diffeomorphism $[M,\ba]\setminus \bS_\ba\times 0  \to 
M\setminus \ba$. 
The pull-backed metric $(\bbl_\ba^M)^*g$ is a singular Riemannian metric over 
$[M,\ba]$ with singular locus $\bS_\ba\times 0$.
In particular, given two Riemannian metrics $g$ and $h$ over~$M$, the spherically blown-up manifolds
$[(M,g),\ba]$ and $[(M,h),\ba]$ are diffeomorphic as manifolds with boundary.
\subsection{Conic metrics}\label{sec:conic}
Let $(M,\dd M)$ be a smooth manifold with compact boundary. 
We present below the classical material of \cite{Che1,Che2} with some additional notions necessary for this paper.

\begin{definition}\label{def:conic-manif}
A \em conic metric $g^c$ \em on the smooth manifold $(M,\dd M)$	is a singular
Riemannian metric over $M$ with $\sing(g^c) = \dd M$,  for which there exists a
smooth collar neighbourhood diffeomorphism
$$
\vp = \phi_r^\eta : M\cap \{r < \eta \} \to \Cyl(\dd M,\eta)
$$
such that
\begin{equation}\label{eq:con-met}
\vp_*g^c = \rd r^2 + r^2 G(r,\by;\rd r, \rd \by)
\end{equation}
and $\by \to G(0,\by)|_{0\times \dd M}$ is a Riemannian metric over $\dd M$. 	
When $g^c$ is a conic metric, the triple $(M,\dd M,g^c)$ is said to be a \em manifold with a conic metric. \em 
\end{definition}
Note that the singular locus of conic metrics lies on the boundary of the manifold 
and if $(S,\dd S)$ is a p-sub-manifold of the manifold with conic metric $(M,\dd 
M,g^c)$, then $(S,\dd S,g^c|_S)$ is a manifold with a conic metric.
\begin{definition}
The pseudo-distance $d^c:=d_{g^c}$ over $M$ induced by the conic metric $g^c$
is the \em conic distance\em. For a sub-manifold $S$ of $M$, the 
pseudo-distance obtained by restricting to $S$ the conic distance $d^c$ is
its outer conic distance whereas the pseudo-distance obtained from the conic 
metric tensor $g^c$ restricted to $S$ is its inner pseudo-distance.
\end{definition}
The following special cases of conic metrics will be of great use:
\begin{definition}
(1) Let $N$ be a compact manifold without boundary and let $\eta >0$. A \em model 
conic metric \em on the cylinder $\Cyl(N,\eta)$ is a singular Riemannian metric 
$g_\mdl^c$ over $\Cyl(N,\eta)$ of the following form
\begin{equation}\label{eq:mod-con-met}
g_\mdl^c = \rd r^2 + r^2 g^\dd(r,\by; \rd \by)
\end{equation}
where $r \to g^\dd(r)$ is a smooth family of Riemannian metrics over $N$.
\\
(2) A \em simple conic metric \em over the manifold $(M,\dd M)$ with compact 
boundary is any conic metric $g^c$ for which there exists  a collar neighbourhood
diffeomorphism $\vp : U_r^\eta \to \Cyl(\dd M,\eta)$ and a Riemannian metric 
$g_{\dd M}$ over $\dd M$ such that
$$
\vp_* g^c = \rd r^2 + r^2\,g_{\dd M}.
$$
\end{definition}
Any model conic  metric on a cylinder  $\Cyl(N,\eta)$ over 
$N=\dd M$  extends as a conic metric onto $(M,\dd M)$ and there is no other type 
of conic metrics. More precisely:
\begin{theorem}[Theorem 1.2, \cite{MeWu}]\label{thm:MelWun}
Let $(M,\dd M,g^c)$ be a manifold with conic metric. There exists a collar 
neighbourhood diffeomorphism $\vp:U \to \Cyl(\dd M,\eta)$ such that $\vp_*g^c$ is 
a model conic metric over $\Cyl(\dd M,\eta)$. 
\end{theorem}
\begin{example}\label{ex:conic}
1) Let $\bbl_\bbo : [\Rn,\bbo] = \bS^{n-1}\times\Rgo \to \Rn$ be the 
spherical blowing-up of $\Rn$. The pull-back $\bbl_\bbo^*(eucl)$ of the 
Euclidean metric $eucl$ over the blown-up space $[\Rn,\bbo]$ is the 
simplest example of a model conic metric 
$$
\bbl_\bbo^*(eucl) = \rd r^2 + r^2 \rd \bu^2.
$$
2) Warped metrics  over 
$\Cyl(N,\eta)$ are model conic metrics, these are singular metrics of the form
$$
\rd r^2 + (rf(r))^2 g_N
$$
where $f$ is a smooth positive function over $[0,\eta)$ and
$g_N$ is a Riemannian metric over $N$, see also \cite{GrLy}.
\end{example}
An essential by-product of the demonstration of 
Theorem~\ref{thm:MelWun} is the following
\begin{corollary}\label{cor:bound-lev}
If $d^c$ is a conic distance, then the function $d^c(-,\dd M)$ is a boundary defining function. In particular there exists a positive $r_0$ such that  each level $\{d^c(-,\dd M) =r\}$ is diffeomorphic to $\dd M$ for any $0\leq r\leq r_0$.
\end{corollary}
Moreover, for any two conic metrics $g$ and $h$ over $(M,\dd M)$ there exists a 
collar neighbourhood $U$ of $\dd M$ such that $g$ and $h$ restricted to $U$ are 
equivalent and thus the identity mapping  $\textup{id}: (U,d_g)\to(U,d_h)$ is 
pseudo-bi-Lipschitz with respect to the conic distances $d_g$ and $d_h$, see 
Section~\ref{section:pseudo-metric}. In particular, we get
\begin{corollary} \label{cor:conicequivalent}
Any two conic metrics over a compact manifold  are equivalent.
\end{corollary}
\begin{propdef}\label{propdef:associated}
Any conic metric $g^c$ on the manifold $M$ has 
(1) \em an associated Riemannian metric on the boundary \em, i.e. a Riemannian 
metric $g^\dd:=G_{|\dd M}$ on $\dd M$ of representation~\eqref{eq:con-met}, and 
(2) \em an associated simple conic metric \em near the boundary $\dd M$,  i.e. a 
simple conic metric $h$ on a neighbourhood $U$ of $\dd M$ such that the singular metrics $g^c$ and $h$ are equivalent over $U$ and have the same associated Riemannian metric on the boundary.
\end{propdef}
\begin{proof}
Let $g^c$ be a conic metric over $(M,\dd M)$. By definition $g^\dd:=G_{|\dd M}$ is Riemannian over $\dd M$. Let  $\vp : U_r^\eta \to 
\Cyl(\dd M,\eta)$ be a collar neighbourhood diffeomorphism of Theorem~\ref{thm:MelWun} such that 
$\vp_*g^c$ is a model conic metric
$$
\vp_* g^c = \rd r^2 + r^2 g^\dd(r).
$$
Thus $g^\dd(0)$ is a Riemannian metric over $\dd M$, and the model conic metric 
$$
g_0^c = \rd r^2 + r^2 g^\dd(0)
$$
is equivalent to $g^c$ over a smaller cylinder $\Cyl(\dd M,\eta')$ for
some positive height $\eta' < \eta$. Since $\vp^* g_0^c$  extends as a
conic metric $g^s$ over $(M,\dd M)$, the proof ends by Corollary~\ref{cor:conicequivalent}.
\end{proof}

The following, though elementary, is a useful variation of the law 
of cosines in the conic context:
\begin{proposition}\label{prop:conic-dist}
Let $N$ be a 
compact smooth manifold without boundary. Let 
$$g^c= \rd r^2 + r^2 g^\dd(r)$$
be a model conic metric on $\Cyl(N,\eta)$. Denote by $d^c$ the conic distance on $\Cyl(N,\eta)$ and by $d_N$ the distance on $N$ induced by $g^\dd(0)$. Then there exist positive constants $A^c,B^c$, depending only on $g^c$ and $\eta$,  
such that
\begin{equation}\label{eq:conic-dist}
A^c (|r-r'| + \min(r,r')\,d_N(\by,\by'))
\;\leq\; d^c(\bx,\bx') \;\leq\;
B^c(|r-r'| + \min(r,r')\,d_N(\by,\by'))
\end{equation}
for all $\bx =(\by,r)$ and $\bx' = (\by,r')$ of $\Cyl(N,\eta)$.
\end{proposition}
\begin{proof}
If $r\to g^\dd(r)$ is constant and equal to a tensor $g_N$ on $N\times \{0\}$, then 
by  
straightforward computation we get
$$
\frac{|r-r'|}{2} + \frac{\min(r,r')}{2}\, d_N(\by,\by') \leq 
 d^c(\bx,\bx') \;\leq\;
|r-r'| + \min(r,r')\,d_N(\by,\by')
$$
for the associated simple conic metric $\rd r^2 + r^2 g_N$. 
The general case is deduced from the constant case above. See also \cite{CoGrMi4}.
\end{proof}
\subsection{Asymptotically conic metrics}\label{section:acm}
Due to considerations similar to the previous Section~\ref{sec:conic} but for 
scattering metrics, see \cite{MeZw,assconic1}, we will give here a simplified 
presentation of asymptotically conic metrics without taking away from  generality, 
compare Proposition-Definition~\ref{propdef:associated} and Remark 
\ref{rk:scattering}.

Let $M$ be a smooth manifold with compact boundary, $N$ a smooth manifold without boundary and $\eta$ a positive number.
\begin{definition}\label{def:asym-con-met}
 An \em asymptotically conic metric   over the cylinder $\Cyl(N,\eta)$ \em  is a smooth Riemannian metric $g^\infty$ on the open cylinder
$\Cyl(N,\eta)^o$ such that the Riemannian metric $r^4 g^\infty$ 
extends as a smooth model conic metric $g^0$ over $\Cyl(N,\eta)$.

The conic metric $g^0$ is the \em conic metric associated with the  
asymptotically conic metric $g^\infty$ \em and over $\Cyl(N,\eta)^o=N_\by \times(0,\infty)_r$  we can write 
$$
g^\infty = \frac{\rd r^2}{r^4} + \frac{g^\dd(r)}{r^2} = \frac{1}{r^4}\, g^0 
$$
where $\Rgo\ni r\to g^\dd(r)$ is a smooth family of Riemannian metrics over $N$.

An \em asymptotically conic  metric (ac-metric) \em on the manifold $(M,\dd M)$ is a Riemannian metric $g^\infty$ on $M\setminus \dd M$ for which there exists  a collar 
neighbourhood diffeomorphism $\vp : U_r^\eta \to \Cyl(\dd M,\eta)$
such that
$
\vp_* g^\infty 
$
is an asymptotically conic metric  over  $\Cyl(\dd M,\eta)$.

The Riemannian metric $r^2g^\infty|_{\dd M} =g^\dd(0)$ on $\dd M$ is 
the \em Riemannian metric on the boundary associated to the ac-metric. \em
\end{definition}
%
%


\begin{remark}\label{rk:scattering}
	An ac-metric is also known as a \em scattering metric \em 
	\cite{MeZw,assconic1}. The standard definition  does not require 
	that it be a conformal factor of a model conic metric, but just of a conic 
	metric. Yet, a change of coordinates over $\Cyl(N,\eta)$  allows this in
	a neighbourhood of the boundary $ N\times 0$, see \cite{assconic1}.
\end{remark}
Let us illustrate the notion with a simple example.
\begin{example}\label{ex:asymp-con}
	 Let $\bbl_\infty : \bS^{n-1}\times\R_{>0} \to \Rn\setminus \bbo$ be the 
	spherical blowing-up of $\Rn$ at infinity, defined by 
	$$
	\bbl_\infty(\bu,R) = \frac{\bu}{R}.
	$$ 
	The pull-back $\bbl_\infty^*(eucl)$ over $\bS^{n-1}\times\R_{>0}$ is the 
	simplest example of an asymptotically conic metric
	$$
	\bbl_\infty^*(eucl) = \frac{1}{r^4}[\rd r^2 + r^2 \rd \bu^2].
	$$
\end{example}
Every ac-metric $g^\infty$ introduces a distance function $d^\infty$ on $M$ such that $d^\infty(\bx,\by)=\infty$ only if $\bx$ or $\by$ lie in $\dd M$ or in different connected components of $M$.
\begin{proposition}\label{prop:asym-conic-dist}
	Let $\Cyl(N,\eta)$ be equipped with an ac-metric $g^\infty$ and the induced distance~$d^\infty$. Let $g^0= \rd r^2 + r^2g^\dd(r) $ be the associated model conic metric and denote $g_N:=g^\dd(0)$.
For points $\bx =(\by,r)$ and $\bx' = (\by,r')$ of $\Cyl(N,\eta)^o$  define
$$
e(\bx,\bx') = \left| \frac{1}{r}-\frac{1}{r'} \right| + 
\min \left( \frac{1}{r}, \frac{1}{r'} \right)\,d_N(\by,\by').
$$
Then there exist positive constants $A^\infty,B^\infty$, depending only on $g^0$ and 
$\eta$, such that
\begin{equation}\label{eq:asym-conic-dist}
A^\infty\, e(\bx,\bx') \,\leq\, d^\infty(\bx,\bx') \,\leq\,
B^\infty e(\bx,\bx')
\end{equation}
for all $\bx,\bx' \in \Cyl(N,\eta)^o$.
\end{proposition}
\begin{proof}
Proof as in Proposition~\ref{prop:conic-dist}, see \cite{CoGrMi4} and \cite{GrOl} as well.
\end{proof}
\subsection{Manifolds with globally conic metrics}

Let $(M,\dd M)$ be a compact smooth manifold.

\begin{definition}\label{def:glob-conic-manif}
A \em globally conic metric on $M$ \em is a conic metric over $M\setminus 
\dd M_\infty$, where $\dd M_\infty$ a finite union of connected components of 
$\dd M$, which is an asymptotically conic metric on some neighbourhood of $\dd 
M_\infty$ in $M$. 

Then the triple $(M,\dd M,g)$ is a \em  manifold with a globally conic metric \em  and the subset $\dd M_\infty$ is the \em boundary at infinity \em of $(M,\dd M, g)$.
\end{definition}
From the point of view of Lipschitz geometry, the following Proposition~\ref{propdiffeobiLip} shows interchangeability of globally conic metrics, similarly to Riemannian structure, compare Section \ref{section:pseudo-metric}.
\begin{proposition}\label{propdiffeobiLip}
Let $(M,\dd M, g)$ and $(N, \dd N, h)$ be two smooth manifolds with globally conic metrics and $\phi:M\to N$ be a diffeomorphism such that  $\phi(\dd M_\infty) = \dd N_\infty$. Then the conic metrics $\phi_*(g)$ and $h$ are equivalent over
$N\setminus \dd N_\infty$. In particular any two globally conic metrics on a compact manifold yielding the same boundary at infinity are equivalent. 
\end{proposition}

\begin{proof}
By Corollary \ref{cor:conicequivalent} any two conic metrics over $\Cyl(N,\eta)$ are equivalent over $\Cyl(N,\ve)$ for some 
positive $\ve\leq \eta$. Thus by definition any two ac-metrics over $\Cyl(N,\eta)$ are equivalent
over $\Cyl(N,\ve)^o$ for some positive $\ve\leq  \eta$.	Since the singular locus and boundary at infinity of the globally conic metrics $\phi_*(g)$ and $h$ are thus equivalent.
\end{proof}

\subsection{Conic inversion}\label{section:conic-inv}

Let $(N,g_N)$ be a connected compact Riemannian manifold without 
boundary. Denote by $$P = N_\by\times(\Rgo)_r$$  the infinite (non-negative) 
cylinder over $N$. 
Let $P^0 = N\times(0,\infty)$ be the open infinite cylinder over 
$N$. 

The \em conic inversion \em of an open cylinder $P^0$ is the mapping 
$$
\iota^0 : P^0 \to P^0, \;\; (\by,r) \mapsto (\by,R) =\left(\by,
\frac{1}{r}\right).
$$ 
Gluing two disjoint 
copies of $P$ along their interiors $P^0$ identified by the 
inversion $\iota^0$ yields the smooth compact manifold with 
boundary
$$
\ovP = N\times [0,\infty].
$$
The inversion 
$\iota^o$ extends onto $\ovP$ as a diffeomorphism $\iota$, inverting 
$\dd (N\times\Rgo)$ and $\dd (N \times (0,\infty])$. 

Consider the following singular Riemannian metrics 
$$
g^0 = \rd r^2 + r^2 g^\dd(r) \;\; {\rm over} \;\; N_\by\times(\Rgo)_r \;\; {\rm 
	and} \;\; g^\infty = \rd R^2 + R^2 g^\dd(R) \;\; {\rm over} \;\; 
N_\by\times(0,\infty]_R;
$$
Both are model conic metrics over their respective domains. 
On $P^0$ we get
$$
r^{-4}\,g^0 = (\iota^0)^* g^\infty,
$$
thus $r^{-4}\,g^0$ is an  asymptotically conic metric  over 
$P^0$. 

Let $d^0$ be the conic distance  over $P^0$ obtained from $g^0$, and let 
$d^\infty$ be the  distance over $P^0$ obtained from
$ (\iota^0)^* g^\infty$. A simple, though essential,
metric property relating conic  with asymptotically conic distances is summarized as follows.
\begin{proposition}\label{prop:inv-dist}
There exist positive constants $A',B'$ such that for all $\bx,\bx' \in P^0$ the
following holds 
\begin{equation}
A' \cdot d^0(\bx,\bx')\;\leq\; r\cdot r'\cdot d^\infty(\bx,\bx') \;\leq \; 
B'\cdot d^0(\bx,\bx').
\end{equation}
\end{proposition}
\begin{proof}
From Estimates \eqref{eq:conic-dist} and 
\eqref{eq:asym-conic-dist} we can assume that
$$
d^0(\bx,\bx') = |r-r'| + \min(r,r')d_N(\by,\by'), \;\; {\rm and} 
\;\; 
d^\infty(\bx,\bx') = |R-R'| + \min(R,R')d_N(\by,\by'),
$$
with $R=r^{-1}$ and $R'=(r')^{-1}$. We immediately obtain
$$
r\cdot r'\cdot d^\infty \geq d^0.
$$Distinguishing the cases $r\geq 2r'$ and $r'\leq r\leq 2r'$, 
we get the other inequality. 
\end{proof}
%
%
%
%
%
%
%
%
%
%
%
%
%
%
%
%
%
%
%
%
%
%
%
%
%
%
%
%
%
%
%
%
%
%
%
%
%
%
\section{Globally conic singular manifolds}\label{section:GCSM}
The interest in  metrically conical singularities in Riemannian 
geometry dates back at least to the seminal works 
\cite{Che1,ChTa1,ChTa2,Che2}. Definition presented in this section is a 
variation on the standard notion of singular manifold with 
metrically conical points. For more on the local differential 
geometry of  conic singular manifolds see \cite{Mel,Gri}.

We will give natural characterizations of (globally) conic 
singular manifolds: by an atlas of isometries in Definition~\ref{def:conic-sing-manif}, isometric resolution in Section~\ref{section:CRofCSM}, as a quotient space in Section~\ref{sec:quotientspace} and by a system of conic charts in Section~\ref{sec:coniccharts}. 
\subsection{Conic singular manifolds}
Recall from 
Section~\ref{section:pseudo-metric} that if  $M$ is a manifold with a
conic metric or an ac-metric, then the metric induces a length distance  $d_M$
on $M\setminus \dd M$. 
\begin{definition}[Conic singular manifold]\label{def:conic-sing-manif}
A complete length space $(E,d)$ is an $n$-dimensional \em conic 
singular manifold \em if there exists a finite subset 
$\Sgm_E$ of $E$, an open covering $\cU$ of
$E$ and smooth $n$-dimensional manifolds $M_U$ with compact 
boundary such that for any $U\in\cU$ one of the following 
is satisfied

(1) there exists an isometry  
$$
\phi_U:(M_U, d_{M_U})\to (U,d),
$$
where $(M_U, g_U)$ is a Riemannian manifold with empty 
boundary

(2)
there exists a point $\ba\in U\cap \Sgm_E$ and a continuous 
surjective mapping $\phi_U:M_U\to U$    such that the restriction
$$
(\phi_U)_{|M_U\setminus \partial M_U}:(M_U\setminus \partial 
M_U, d_{M_U})\to (U\setminus\ba,d)
$$
is an isometry, where $(M_U,\dd M_U, g_U)$ is a manifold with a conic metric

The smallest among subsets $\Sgm_E$ satisfying the definition is called 
the \em set  $E_\sing$  of  singular points
of $E$. \em Its complement $E^o$ is the set of \em smooth
points of $E$. \em 

Moreover, a conic singular manifold is a \em globally conic
singular manifold \em if additionally

(3) there exist  a compact subset $K$ of $E$ such that
$E\setminus K = U\in\cU$, and an isometry 
$$
\phi_U:(M_U\setminus \partial M_U, d_{M_U})\to (U,d),
$$
where $(M_U,\dd M_U,g_U)$ is a manifold with  an asymptotically conic metric.
\end{definition}
\tikzset{every picture/.style={line width=0.75pt}} 
\begin{figure}[h]
\begin{tikzpicture}[x=0.75pt,y=0.75pt,yscale=-0.7,xscale=0.7]
	\centering

\draw    (87.95,13.19) .. controls (142.99,0.43) and (159.38,85.91) .. (205.05,27.22) ;
\draw    (102,270.91) .. controls (92.63,139.5) and (-23.31,47.64) .. (87.95,13.19) ;
\draw    (165.24,268.36) .. controls (155.87,83.36) and (231.99,131.84) .. (205.05,27.22) ;
\draw  [dash pattern={on 4.5pt off 4.5pt}] (102,270.91) .. controls (102,262.46) and (116.16,255.6) .. (133.62,255.6) .. controls (151.08,255.6) and (165.24,262.46) .. (165.24,270.91) .. controls (165.24,279.37) and (151.08,286.22) .. (133.62,286.22) .. controls (116.16,286.22) and (102,279.37) .. (102,270.91) -- cycle ;
\draw   (86.48,92.29) .. controls (79.37,48.61) and (80.02,13.19) .. (87.95,13.19) .. controls (95.87,13.19) and (108.06,48.61) .. (115.17,92.29) .. controls (122.29,135.98) and (121.63,171.4) .. (113.71,171.4) .. controls (105.79,171.4) and (93.6,135.98) .. (86.48,92.29) -- cycle ;
\draw    (483.77,23.4) .. controls (497.82,45.09) and (561.06,93.57) .. (580.97,28.5) ;
\draw    (467.37,272.19) .. controls (421.7,152.26) and (346.75,159.91) .. (371.35,27.22) ;
\draw    (530.61,269.64) .. controls (521.24,84.64) and (643.04,133.12) .. (616.1,28.5) ;
\draw   [color={rgb, 255:red, 144; green, 19; blue, 254 }  ,draw opacity=1 ][line width=1]  (467.37,272.19) .. controls (467.37,263.73) and (481.53,256.88) .. (498.99,256.88) .. controls (516.46,256.88) and (530.61,263.73) .. (530.61,272.19) .. controls (530.61,280.65) and (516.46,287.5) .. (498.99,287.5) .. controls (481.53,287.5) and (467.37,280.65) .. (467.37,272.19) -- cycle ;
\draw  [color={rgb, 255:red, 144; green, 19; blue, 254 }  ,draw opacity=1 ][line width=1]  (371.35,25.32) .. controls (371.35,17.21) and (381.57,10.64) .. (394.18,10.64) .. controls (406.79,10.64) and (417.02,17.21) .. (417.02,25.32) .. controls (417.02,33.43) and (406.79,40) .. (394.18,40) .. controls (381.57,40) and (371.35,33.43) .. (371.35,25.32) -- cycle ;
\draw  [color={rgb, 255:red, 144; green, 19; blue, 254 }  ,draw opacity=1 ][line width=1]  (438.1,23.4) .. controls (438.1,16) and (448.32,10) .. (460.93,10) .. controls (473.55,10) and (483.77,16) .. (483.77,23.4) .. controls (483.77,30.8) and (473.55,36.79) .. (460.93,36.79) .. controls (448.32,36.79) and (438.1,30.8) .. (438.1,23.4) -- cycle ;
\draw  [color={rgb, 255:red, 144; green, 19; blue, 254 }  ,draw opacity=1 ][line width=1]  (580.97,28.5) .. controls (580.97,22.16) and (588.83,17.02) .. (598.54,17.02) .. controls (608.24,17.02) and (616.1,22.16) .. (616.1,28.5) .. controls (616.1,34.84) and (608.24,39.98) .. (598.54,39.98) .. controls (588.83,39.98) and (580.97,34.84) .. (580.97,28.5) -- cycle ;
\draw    (417.02,27.22) .. controls (460.35,324.5) and (458.01,117.81) .. (438.1,23.4) ;
\draw    (354.95,158.64) -- (241.02,158.64) ;
\draw [shift={(239.02,158.64)}, rotate = 360] [color={rgb, 255:red, 0; green, 0; blue, 0 }  ][line width=0.75]    (10.93,-3.29) .. controls (6.95,-1.4) and (3.31,-0.3) .. (0,0) .. controls (3.31,0.3) and (6.95,1.4) .. (10.93,3.29)   ;

\end{tikzpicture}
\caption{A globally conic singular manifold and its smooth model.}

\end{figure}
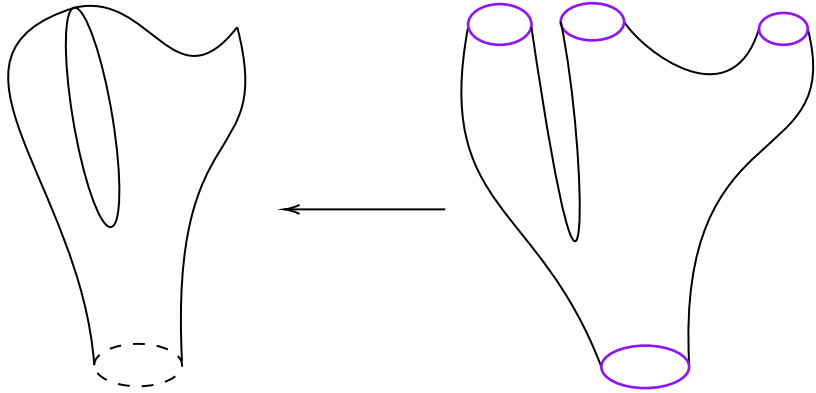

 Note that the transition mappings $\phi_U^{-1}\circ\phi_V$, appropriately restricted,  are
smooth Riemannian isometries  
(see \cite{MySt,Pal}), in particular smooth diffeomorphisms. Thus we get the following  Proposition~\ref{prop:rksonconicsing}.

\begin{proposition}\label{prop:rksonconicsing}
	If $(E,d)$ is a (globally) conic singular manifold, then 
	\\
	(1) The set of smooth points $E^o$ is a smooth manifold.
	\\
	(2) The Riemannian metrics $(\phi_U)_* (g_{U|M_U\setminus\dd M_U})$ glue to a Riemannian metric $g^o$ over $E^o$.
	\\
	(3) The metric spaces $(E^o,d^o)$ and $(E^o,d_{|E^o})$ are locally isometric, where $d^o$ is  the distance function  induced by the Riemann tensor $g^o$ on $E^o$.
\end{proposition}

\subsection{Conic singular manifolds as quotient spaces} \label{sec:quotientspace}
Let $M$ be a manifold with conic metric $g$ and denote $M^o:= M\setminus \dd M$. Recall that the conic tensor induces a pseudo-distance $d^c$ on $M$ which is infinite between points in different connected components of $M$ and is equal zero between points in the same boundary component.

Let $\Sigma$ be a finite set and $\sigma: M\to M^o\cup \Sigma$ a surjective map 
such that $\sigma_{|M^o}=\textup{id}_{M^o}$ and for each boundary component $N$ of 
$\dd M$ we have $\sigma(N)\in \Sigma$. Let $M^\sigma$ denote the quotient space 
obtained from $M$ by the mapping $\sigma$. With a slight abuse of notation we
will say that the mapping $\sigma: M\to M^\sigma$ is a \em boundary collapsing 
map.\em 

Let us introduce distance on $M^\sigma$ as 
$$
d^\sigma(\bx, \by) := \min \left(  d^c(\bx, \by), d^c(\bx, \sigma^{-1}(p_0)) + \sum_{j=1}^k d^c( \sigma^{-1}(p_{j-1}), \sigma^{-1}(p_j) )  + d^c(\sigma^{-1}(p_k),\by)  \right),
$$
where minimum is taken with respect to all chains of points $p_0,\dots, p_k$ in $\Sigma$. 
\begin{figure}[h]
	\tikzset{every picture/.style={line width=0.75pt}} 
	
	\begin{tikzpicture}[x=0.75pt,y=0.75pt,yscale=-0.7,xscale=0.7]
	\centering 	
\draw    (88.95,88.19) .. controls (143.99,75.43) and (160.38,160.91) .. (206.05,102.22) ;
\draw    (103,345.91) .. controls (93.63,214.5) and (-22.31,122.64) .. (88.95,88.19) ;
\draw    (166.24,343.36) .. controls (156.87,158.36) and (232.99,206.84) .. (206.05,102.22) ;
\draw  [dash pattern={on 4.5pt off 4.5pt}] (103,345.91) .. controls (103,337.46) and (117.16,330.6) .. (134.62,330.6) .. controls (152.08,330.6) and (166.24,337.46) .. (166.24,345.91) .. controls (166.24,354.37) and (152.08,361.22) .. (134.62,361.22) .. controls (117.16,361.22) and (103,354.37) .. (103,345.91) -- cycle ;
\draw    (484.77,98.4) .. controls (498.82,120.09) and (562.06,168.57) .. (581.97,103.5) ;
\draw    (468.37,347.19) .. controls (422.7,227.26) and (347.75,234.91) .. (372.35,102.22) ;
\draw    (531.61,344.64) .. controls (522.24,159.64) and (644.04,208.12) .. (617.1,103.5) ;
\draw  [dash pattern={on 4.5pt off 4.5pt}] (468.37,347.19) .. controls (468.37,338.73) and (482.53,331.88) .. (499.99,331.88) .. controls (517.46,331.88) and (531.61,338.73) .. (531.61,347.19) .. controls (531.61,355.65) and (517.46,362.5) .. (499.99,362.5) .. controls (482.53,362.5) and (468.37,355.65) .. (468.37,347.19) -- cycle ;
\draw  [color={rgb, 255:red, 144; green, 19; blue, 254 }  ,draw opacity=1 ][line width=1]  (372.35,100.32) .. controls (372.35,92.21) and (382.57,85.64) .. (395.18,85.64) .. controls (407.79,85.64) and (418.02,92.21) .. (418.02,100.32) .. controls (418.02,108.43) and (407.79,115) .. (395.18,115) .. controls (382.57,115) and (372.35,108.43) .. (372.35,100.32) -- cycle ;
\draw  [color={rgb, 255:red, 74; green, 144; blue, 226 }  ,draw opacity=1 ][line width=1]  (439.1,98.4) .. controls (439.1,91) and (449.32,85) .. (461.93,85) .. controls (474.55,85) and (484.77,91) .. (484.77,98.4) .. controls (484.77,105.8) and (474.55,111.79) .. (461.93,111.79) .. controls (449.32,111.79) and (439.1,105.8) .. (439.1,98.4) -- cycle ;
\draw  [color={rgb, 255:red, 155; green, 155; blue, 155 }  ,draw opacity=1 ][line width=1]  (581.97,103.5) .. controls (581.97,97.16) and (589.83,92.02) .. (599.54,92.02) .. controls (609.24,92.02) and (617.1,97.16) .. (617.1,103.5) .. controls (617.1,109.84) and (609.24,114.98) .. (599.54,114.98) .. controls (589.83,114.98) and (581.97,109.84) .. (581.97,103.5) -- cycle ;
\draw    (418.02,102.22) .. controls (461.35,399.5) and (459.01,192.81) .. (439.1,98.4) ;
\draw    (355.95,233.64) -- (242.02,233.64) ;
\draw [shift={(240.02,233.64)}, rotate = 360] [color={rgb, 255:red, 0; green, 0; blue, 0 }  ][line width=0.75]    (10.93,-3.29) .. controls (6.95,-1.4) and (3.31,-0.3) .. (0,0) .. controls (3.31,0.3) and (6.95,1.4) .. (10.93,3.29)   ;
\draw    (58,103) .. controls (95,297) and (169,311) .. (106,88) ;
\draw    (58,103) .. controls (90,139) and (105,104) .. (106,88) ;
\draw  [color={rgb, 255:red, 144; green, 19; blue, 254 }  ,draw opacity=1 ][line width=1]  (372.35,65.32) .. controls (372.35,57.21) and (382.57,50.64) .. (395.18,50.64) .. controls (407.79,50.64) and (418.02,57.21) .. (418.02,65.32) .. controls (418.02,73.43) and (407.79,80) .. (395.18,80) .. controls (382.57,80) and (372.35,73.43) .. (372.35,65.32) -- cycle ;
\draw  [color={rgb, 255:red, 74; green, 144; blue, 226 }  ,draw opacity=1 ][line width=1]  (435.1,63.4) .. controls (435.1,56) and (445.32,50) .. (457.93,50) .. controls (470.55,50) and (480.77,56) .. (480.77,63.4) .. controls (480.77,70.8) and (470.55,76.79) .. (457.93,76.79) .. controls (445.32,76.79) and (435.1,70.8) .. (435.1,63.4) -- cycle ;
\draw    (372.35,65.32) .. controls (374,20) and (481,24) .. (480.77,63.4) ;
\draw    (418.02,65.32) .. controls (423,55) and (429.1,53.4) .. (435.1,63.4) ;
\draw [line width=1.5]  [dash pattern={on 1.69pt off 2.76pt}]  (58,103) .. controls (94,102) and (92,92) .. (106,88) ;
\draw [line width=1.5]  [dash pattern={on 1.69pt off 2.76pt}]  (58,103) .. controls (51,116) and (51,129) .. (55,138) ;
\draw [line width=1.5]  [dash pattern={on 1.69pt off 2.76pt}]  (106,88) .. controls (118,95) and (126,102) .. (132,114) ;
\draw [line width=1.5]  [dash pattern={on 1.69pt off 2.76pt}]  (395.18,115) .. controls (394,126) and (391,140) .. (392.18,150) ;
\draw [line width=1.5]  [dash pattern={on 1.69pt off 2.76pt}]  (395.18,50.64) .. controls (414,39) and (452,43) .. (457.93,50) ;
\draw [line width=1.5]  [dash pattern={on 1.69pt off 2.76pt}]  (461.93,111.79) .. controls (468,126) and (466,125) .. (472,137) ;

\draw (43,134) node [anchor=north west][inner sep=0.75pt]   [align=left] {x};
\draw (377,138) node [anchor=north west][inner sep=0.75pt]   [align=left] {x};
\draw (122,114) node [anchor=north west][inner sep=0.75pt]   [align=left] {y};
\draw (476,131) node [anchor=north west][inner sep=0.75pt]   [align=left] {y};

\end{tikzpicture}
\caption{Distance between points in the quotient space.}
\end{figure}
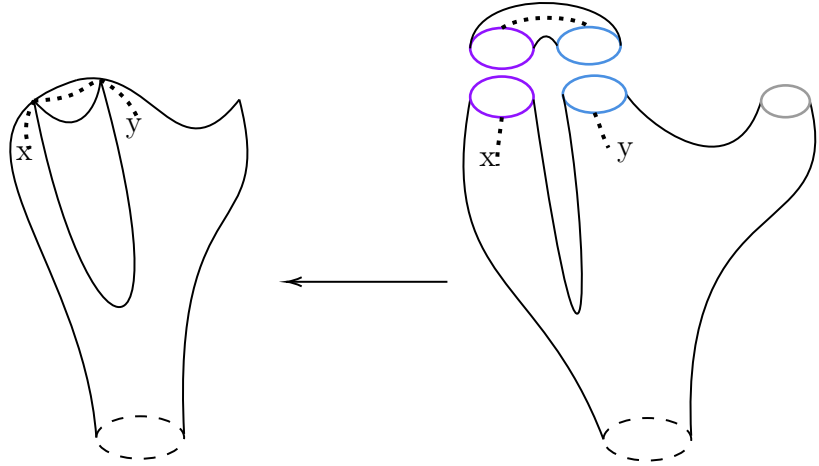
\begin{proposition}[Conic singular manifold as quotient space]\label{propQuotient}
Let $\sigma:M\to M^\sigma$ be a boundary collapsing map on a manifold $M$ with a conic metric. Then $(M^\sigma, d^\sigma)$ is a conic singular manifold. If 
additionally, the manifold $M$ is densely embedded in a compact manifold  
with a globally conic metric, then $(M^\sigma, d^\sigma)$ is a globally conic 
singular manifold. 
\end{proposition}
\begin{proof}
It suffices to note that the space $(M^\sigma, d^\sigma)$ is a length space, since  $d^\sigma$ is a distance 
function and the conic distance between a point of $M$ and a 
component of its compact boundary is either realized along
a smooth arc or infinite.
\end{proof}
Note that a boundary collapsing map $\sigma$ can be alternatively described as a composition of metric identification on each connected component of $M$ and then glueing points between different connected components, the choice of which points to glue depends on $\sigma$ and may produce topologically distinct spaces.

\subsection{Conic charts }\label{sec:coniccharts}

\begin{propdef}\label{propdef:coniccharts}
Let $(E,d)$ be a complete length space such that $E^o$, 
the complement of a finite subset $\Sgm$, is a Riemannian 
manifold without boundary with metric tensor~$g^o$.
	
The space $(E,d)$ is a  conic singular manifold of dimension 
$n$ if and only if every point $\ba$ of $E$ admits a \em  conic 
chart $\varphi_\ba$ centred at $\ba$, \em i.e. a 
continuous surjective mapping
$$ \vp_\ba: \Cyl(N_\ba,\eta) \to U_\ba ,$$
such that 
\\
(1) $U_\ba$ is an open neighbourhood of $\ba$,
\\
(2) $\Cyl(N_\ba,\eta)$ is an $n$-dimensional cylinder over a smooth compact 
manifold without boundary~$N_\ba$,
\\
(3) $(\vp_\ba)_{| \Cyl(N_\ba,\eta)^o}: \Cyl(N_\ba,\eta)^o
\to U\setminus \ba$ is a diffeomorphism,
\\
(4) $ \vp_\ba^*g^o $ extends smoothly as a conic metric 
over $\Cyl(N_\ba,\eta)$.

Moreover, $(E,d)$ is a globally conic singular manifold if and 
only if additionally  there is an \em  asymptotically conic 
chart \em 
$$
\vp_\infty:  \Cyl(N_\infty,\eta)^o\to U_\infty,
$$
where
\\
(1') $U_\infty$ is a complement of a compact subset 
of $E$,
\\
(2') $\Cyl(N_\infty,\eta)$ is an $n$-dimensional cylinder over a compact smooth manifold without boundary~$N_\infty$
\\
(3') $\vp_\infty$ is a diffeomorphism,
\\
(4') $ \varphi_\infty^*g^o $ is an ac-metric over  $\Cyl(N_\infty,\eta)^o$.

A conic chart is a \em (asymptotically) model conic chart \em 
if the metric in point (4) (resp. in point (4')) is model.
\end{propdef}
\begin{proof}
We check those conditions imply the (globally) conic 
singular nature of $(E,d)$ using the pull-backs of $g^o$.
For $\ba\in E_\sing$, let $\vp_\ba$ be the mapping 
$\phi_U$ of point $(2)$ of Definition \ref{def:conic-sing-manif}. For 
$\ba\in U\subset E^o$, the mapping $\phi_U\circ\bbl_{\ba_U}$, 
where $\ba_U = \phi_U^{-1}(\ba)$, will satisfy the requirements.
\end{proof}
Using conic charts, Corollary \ref{cor:bound-lev} 
can be rephrased as
\begin{lemma}\label{lem:bound-lev}
Under assumptions and notation of Proposition-Definition~\ref{propdef:coniccharts} let $d_\ba:E\to\R$ be the distance function to $\ba  $ in the singular conic manifold $(E,d)$.  

Possibly shrinking $U_\ba$, the function $\vp_\ba^*d_\ba$
extends as a boundary defining function in $\Cyl(N_\ba,\eta)$. Moreover, if $\ba\in E\setminus U_\infty$, then possibly shrinking $U_\infty$, the function
$\vp_\infty^*(\frac{1}{d_\ba})$ extends as a boundary defining 
function in $\Cyl(N_\infty,\eta)$.
\end{lemma}
\begin{proposition}\label{prop:conicchartsunique}
Let $(E,d)$ be a (resp. globally) conic singular manifold with two 
(resp. asymptotically) conic charts  $\vp_\ba: \Cyl(N_\ba,\eta)\to U_\ba$
and $\vp_\ba':  \Cyl(N_\ba',\eta) \to U_\ba'$ at a point $\ba\in E$ (resp. $\ba=\infty$).
Then $N_\ba$ is diffeomorphic to $N_\ba'$.
\end{proposition}
\begin{proof}
Since  an appropriate restriction of the mapping $\vp_\ba$ is a Riemannian
isometry between $\Cyl(N_\ba,\eta)^o$ and $ U_\ba\setminus \ba$, for $0<\ve<\eta$ 
with $\eta$ small enough, we find that $(\vp_\ba^*d_\ba)^{-1} (\ve) $, 
$((\vp_\ba')^*d_\ba)^{-1} (\ve) $ and $(d_\ba)^{-1}  (\ve)$ are isometric 
manifolds by Lemma \ref{lem:bound-lev}, using collar neighbourhood 
diffeomorphisms. Since the levels of a boundary defining function for values small 
enough are diffeomorphic to the boundary, we get the claim.
\end{proof}
Therefore, by Proposition \ref{prop:conicchartsunique} and 
Theorem~\ref{thm:MelWun} our definition of conic charts is not 
restrictive when we ask the conic metric to be model.
\begin{corollary}
A complete length space is a conic singular manifold if and only 
if every point admits model conic charts.
\end{corollary}
\subsection{Conic resolution and completion}\label{section:CRofCSM}
From definition of conic singular manifolds and Proposition~\ref{prop:rksonconicsing} we obtain immediately the following   which	can be treated as a converse to Proposition \ref{propQuotient}.  

\begin{proposition}[Isometric resolution of conic singularities]\label{propGLOBALchart} A complete length space $(E,d)$ is a conic 
singular manifold if and only if there exists a manifold $M$ with conic metric $g$ 
and a continuous surjective mapping 
$$
\phi: M   \to E
$$
such that $\phi(\partial M) = E_\sing$ and the restriction $\phi_{|M\setminus 
\partial M}$ is a Riemannian isometry over $E\setminus E_\sing$. 
	
A complete length space $(E,d)$ is a globally conic singular manifold if and only 
if additionally there exists a compact manifold $\Mbar$ such that
$M  = \Mbar\setminus \dd M_\infty$ for some union of boundary components 
$\dd M_\infty$ of $\Mbar$ and  $g $ is a globally conic metric on 
$\Mbar$. 
\end{proposition}
Given a (globally) conic singular manifold $E$, a pair $(M,\phi)$ as in Proposition~\ref{propGLOBALchart} will be referred to as a \em (global) conic resolution of $E$ with boundary at infinity~$\dd 
M_\infty$. \em
A (global) conic resolution s essentially unique as the following shows.
\begin{proposition}\label{prop:unique-resol}
Let $(M_p,\dd M_p, g_p^c)$ be two (resp. globally) conic manifolds for $p=1,2$
such that there exist mappings $\phi_p:M_p\to E$ satisfying 
Proposition~\ref{propGLOBALchart}. Then  for $p=1,2$ the manifolds $(M_p,\dd M_p)$ 
are diffeomorphic and $M_p\setminus \dd M_p$ are isometric. Moreover, the 
boundaries $\dd M_p$ (resp. and boundaries at infinity $\dd M_{p,\infty}$) are 
isometric with respect to the  Riemannian metric on the boundary associated to the 
(resp. asymptotically) conic metric.
\end{proposition}
\begin{proof}
Let $M_p^o = M_p\setminus\dd M_p$ and 
$\phi_p^o = \phi_p|_{M_p^o}$. Thus $(\phi_2^o)^{-1}\circ
\phi_1^o : M_1^o\to M_2^o$ is a Riemannian isometry.	
\\
Now, let $N = \cup_{\ba\in E_\sing} N_\ba$. Set 
$$
U_\eta : = \{d(-,E_\sing) <\eta\} \;\; {\rm and} \;\; 
\Sgm_\eta :=  \{d(-,E_\sing) =\eta\}.
$$
If $\eta$ is small enough, we define the mapping
$\vp_\eta: \Cyl(N,\eta) \to U_\eta$ as the union $\sqcup_{\ba\in E_\sing} \vp_\ba$
of the conic charts centred at $E_\sing$. Lemma \ref{lem:bound-lev} implies that 
$\Sgm_\ve$ is diffeomorphic to $N$ for $0<\ve<\eta$. Thus $(\phi_p^o)^{-1}\circ
\vp_\eta: \Cyl(N,\eta)\setminus N\to U_\eta^p\setminus\dd M_p$ is an isometry, where 
$$
U_\eta^p = \phi_p^{-1}(U_\eta) = \{d_p^c(-,\dd M_p) <\eta\} \;\; {\rm and} \;\;
\Sgm_\eta^p :=  \phi_p^{-1}(\Sgm_\eta) =  \{d_p^c(-,\dd M_p) = \eta\}.
$$
Thus $\phi_p^*(d(-,E_\sing))$ extends as a boundary defining function of $\dd M_p$ by Lemma \ref{lem:bound-lev}, so that we deduce $\dd M_p = N$. This is enough to guarantee that 
$U_\eta^p$ is isometric to $\Cyl(N,\eta)$ via a collar neighbourhood diffeomorphism of $\dd M_p$ and a smaller $\eta$. Then we glue isometrically.	
\\
Isometry on the boundary follows from consideration as in Proposition~\ref{prop:conicchartsunique}.
\end{proof}
\begin{propdef}
For any globally conic singular manifold $(E,d)$ there exists a compact conic 
singular manifold $(\Ebar,\dbar)$ called a \em conic completion \em of $(E,d)$ 
such that
\\
(1) the space $E$ is topologically embedded in $\Ebar =E \sqcup \infty_E$, 
where the \em set of conic ends $\infty_E$ of $E$ \em is finite
\\
(2) the distances $d$ and $\dbar$ are equivalent  over any compact subset 
of $E$
\\
(3) there exists a neighbourhood $U$ of $E_\sing$ such that $d_{|U}=\dbar_{|U}$
\\
(4) the global conic resolution space $\Mbar=M\cup \dd M_\infty$ of $(E,d)$ is diffeomorphic to the conic resolution space of $(\Ebar,\dbar)$ 
\\
(5) the Riemannian tensor $\overline{g}^o$ on $\Ebar^o$ has the same 
associated Riemannian metric on $\dd M_\infty$ as the tensor $g^o$ of $E^o$
\end{propdef}

\begin{proof}
Let $(E,d)$ be a globally conic singular manifold and $\Mbar=M\cup 
\dd M_\infty$ be its global conic resolution with the globally conic metric 
$g$. By Proposition~\ref{propGLOBALchart} the topological space $E$ 
embedds isometrically as an open dense subset into $\overline{E}:=E\cup \dd 
M_\infty$. By Proposition~\ref{prop:unique-resol} the set $\dd M_\infty$ is 
unique up to 
diffeomorphism and has $k$ connected components. If $\sigma:M\to E$ is the 
boundary collapsing function of $E$, then let us extend it as a boundary 
collapsing function ${\sigma}:\Mbar\to E\sqcup\{p_1,\dots, p_k\}$ so that 
each boundary component of $\dd M_\infty$ maps to a different point 
$p_j\in\infty_E:=\{p_1,\dots, p_k\}$. 
	
Let $\overline{g}$ be the conic metric on $\Mbar$ obtained as a smooth glueing of the conic metric $g$ restricted to a neighbourhood of $\dd M$ with a conic metric $h$ over $\Mbar$, having the same Riemannian metric associated on $\dd M_\infty$ as the ac-metric $g$ near $\dd M_\infty$.

Let now $\overline{E}$ be equipped with the distance 
$\overline{d}:=d^{\overline{\sigma}}$ constructed from 
$(\Mbar,\overline{g})$ as in Section~\ref{sec:quotientspace}. 
	
Claim (1) is obvious. Claim (2) follows from Corollary \ref{cor:conicequivalent} 
and claims (3)  from choice of boundary collapsing function and the tensor 
$\overline{g}$. Point (4) is obvious since by construction $\Mbar$ is the 
resolution space for both $E$ and $\overline{E}$ which is unique up to 
diffeomorphism by Proposition~\ref{prop:unique-resol}. Thus (5) follows from (4) 
and the construction of $\overline{g}$.
\end{proof}
A conic completion can be considered as a compactification of each end of $E$
by one point, and finding an appropriate bounded metric on this compactification. 
Note that a conic completion can be a smooth manifold, for instance the 
$n$-dimensional sphere with standard distance is a conic completion of $\Rn$, so 
although $\overline{E}_\sing$ contains $E_\sing$, it may differ from $E_\sing\cup 
\infty_E$.

The conic completion is unique up to bi-Lipschitz homeomorphism as the following 
shows.
\begin{proposition}\label{prop:conic-complet}
For any two conic completions  $(\Ebar_p,\dbar_p)$, $p=1,2$, of a globally conic 
singular manifold $(E,d)$ there exists a bi-Lipschitz homeomorphism 
$(\Ebar_1,\dbar_1)\to (\Ebar_2,\dbar_2)$, which is a diffeomorphism $\Ebar_1
\setminus \Ebar_{1,\sing}$ onto $\Ebar_2\setminus \Ebar_{2,\sing}$.
\end{proposition}
\begin{proof}
For $p=1,2$ let $(\Mbar_p,\dd \Mbar_p)$ be globally conic resolution spaces 
of $(E,d)$ which yield the conic completions $\Ebar_p$ when equipped with 
conic metrics $h_p$ by the boundary collapsing maps $\sigma_p:\Mbar_p\to \Ebar_p$. Let $\psi:(\Mbar_1, \dd \Mbar_1)\to(\Mbar_2,\dd \Mbar_2)$ be the 
diffeomorphism of Proposition \ref{prop:unique-resol}. Since $\Psi$ is an 
isometry outside the boundary, we necessarily get  $\psi(\dd \Mbar_{1,\infty})
= \dd \Mbar_{2,\infty}$ and the mapping $\Psi:= \sigma_2\circ \psi\circ 
\sigma_1$ yields a homeomorphism $(\Ebar_1,\dbar_1)\to (\Ebar_2,\dbar_2)$, 
which is a diffeomorphism outside the singular points. Since $\psi_*(h_1)$ is 
a conic metric over $(\Mbar_2,\dd \Mbar_2)$, by Corollary 
\ref{cor:conicequivalent} it is bi-pseudo-Lipschitz with respect to conic 
distances introduced by the conic metrics $h_p$. By the quotient construction 
of conic completions, the mapping $\Psi$ is bi-Lipschitz. If a point of 
$\Ebar_p\setminus \infty_E$ is smooth, then $\Psi $ extends as a 
diffeomorphism through this point, thus $\Psi$ is a diffeomorphism $\Ebar_1\setminus \Ebar_{1,\sing}$ onto $\Ebar_2\setminus \Ebar_{2,\sing}$.
\end{proof}
In this context we obtain the global version of Proposition \ref{prop:inv-dist} as 
follows, which is a central result of the paper.
\begin{corollary}\label{cor:conic-complet}
Let $(\overline{E},\overline{d})$ be a conic completion of a globally conic 
manifold $(E,d)$. There exist positive constants $A,B$ 
such that 
$$
A\,\overline{d} (\bx,\bx')\;\leq\; r\,r'\,d(\bx,\bx') \;\leq \; 
B\,\overline{d}(\bx,\bx')  , \;\; 
$$
for $\bx, \bx'\in E$ with $r := 
\overline{d}(\bx,\infty_E)$ and $r' := \overline{d}(\bx',\infty_E)$. 
\end{corollary}
\begin{proof} 
Let $\vpbar: \ovP = \Cyl(\Nbar,\etabar)\to \Ubar$ be the union of conic charts of 
$\overline{E}$ centred at all the points of $\infty_E$. Let $\gbar^c$ be the conic 
metric over $\ovP$ consistent with the pull-back $\vpbar^*(\dbar|_{\Ubar})$. 
We can assume that $\vp$, the restriction of $\vpbar$ to $\ovP^o$ is an asymptotic 
chart of $E$. Let $g^\infty$ be the asymptotically conic metric with distance function equal to the pull-back $\vp^*(d|_{E\cap \Ubar})$. Therefore 
the desired estimates holds true in $E\cap \Ubar$ by Proposition 
\ref{prop:inv-dist} used with $\dbar^c$ and $d^\infty$, pseudo-distances induced 
respectively from $\gbar^c$ and $g^\infty$, thus holds in $V=\clos_E(\Ubar)$. 
Since $K = \Ebar \setminus \Ubar$ is a compact subset of $E$ containing a 
neighbourhood of 
$E_\sing$, both $d,\dbar$ are bounded over $K$ and thus $r,r,'$ are bounded above 
and below over $K$, therefore the estimates hold over $K$. Let $\bx\in K$ and 
$\bx'\in \Ubar$. Let $\by\in K\cap V$ such that $d(\bx,\bx') = d(\bx,\by) + 
d(\by,\bx')$. Thus we deduce from the estimates in $K$ and in
$V$ that $A\,\dbar(\bx,\bx')\leq r\cdot r'\cdot d(\bx,\bx')$. The 
remaining estimate is obtained similarly.
\end{proof}
\subsection{Globally conic singular sub-manifolds}
The following definition coincides with the standard 
definition, for instance for spherical blowings-up, compare~\cite{CoGrMi3}.
Let $(E,d)$ be a (globally) conic singular 
manifold. The \em strict transform of a closed subset $X$ of $E$ by a conic chart $\varphi_\ba$  \em for $\ba\in E\cup\{\infty\}$ is the subset
of $\Cyl(N_\ba,\eta)$ defined as
$$
\stt_\ba 
X:=\clos_{\Cyl(N_\ba,\eta)}(\varphi_\ba^{-1}(X\setminus\ba)).
$$
\begin{definition}
Let $X$ be a closed subset of a conic singular manifold $E$ such that 
outside of a  finite subset it is a sub-manifold of $E^o$. The subset $X$ is:
\\
(1) \em a conic singular sub-manifold  \em of $E$ if its strict 
transform by every conic chart  is a p-sub-manifold
nearby the boundary of the conic chart domain, 
\\
(2) \em an asymptotically conic singular sub-manifold  \em of the globally conic singular manifold $E$ if its 
strict transform by the asymptotic conic chart is a p-sub-manifold 
nearby the boundary  of the conic chart domain,
\\
(3) \em a globally conic singular sub-manifold  \em of the globally conic singular manifold $E$ if 
it is both conic singular and asymptotically conic singular sub-manifold.
\end{definition}
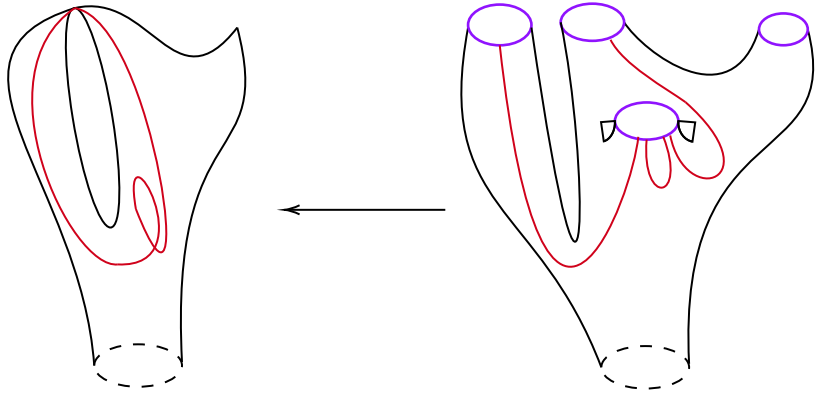
\begin{figure}[h]
\tikzset{every picture/.style={line width=0.75pt}} 
	
\begin{tikzpicture}[x=0.75pt,y=0.75pt,yscale=-0.7,xscale=0.7]
\centering 

\tikzset{every picture/.style={line width=0.75pt}} 

	\draw    (87.95,13.19) .. controls (142.99,0.43) and (159.38,85.91) .. (205.05,27.22) ;
	\draw    (102,270.91) .. controls (92.63,139.5) and (-23.31,47.64) .. (87.95,13.19) ;
	\draw    (165.24,268.36) .. controls (155.87,83.36) and (231.99,131.84) .. (205.05,27.22) ;
	\draw  [dash pattern={on 4.5pt off 4.5pt}] (102,270.91) .. controls (102,262.46) and (116.16,255.6) .. (133.62,255.6) .. controls (151.08,255.6) and (165.24,262.46) .. (165.24,270.91) .. controls (165.24,279.37) and (151.08,286.22) .. (133.62,286.22) .. controls (116.16,286.22) and (102,279.37) .. (102,270.91) -- cycle ;
	\draw   (86.48,92.29) .. controls (79.37,48.61) and (80.02,13.19) .. (87.95,13.19) .. controls (95.87,13.19) and (108.06,48.61) .. (115.17,92.29) .. controls (122.29,135.98) and (121.63,171.4) .. (113.71,171.4) .. controls (105.79,171.4) and (93.6,135.98) .. (86.48,92.29) -- cycle ;
	\draw    (483.77,23.4) .. controls (497.82,45.09) and (561.06,93.57) .. (580.97,28.5) ;
	\draw    (467.37,272.19) .. controls (421.7,152.26) and (346.75,159.91) .. (371.35,27.22) ;
	\draw    (530.61,269.64) .. controls (521.24,84.64) and (643.04,133.12) .. (616.1,28.5) ;
	\draw  [dash pattern={on 4.5pt off 4.5pt}] (467.37,272.19) .. controls (467.37,263.73) and (481.53,256.88) .. (498.99,256.88) .. controls (516.46,256.88) and (530.61,263.73) .. (530.61,272.19) .. controls (530.61,280.65) and (516.46,287.5) .. (498.99,287.5) .. controls (481.53,287.5) and (467.37,280.65) .. (467.37,272.19) -- cycle ;
	\draw  [color={rgb, 255:red, 144; green, 19; blue, 254 }  ,draw opacity=1 ][line width=1]  (371.35,25.32) .. controls (371.35,17.21) and (381.57,10.64) .. (394.18,10.64) .. controls (406.79,10.64) and (417.02,17.21) .. (417.02,25.32) .. controls (417.02,33.43) and (406.79,40) .. (394.18,40) .. controls (381.57,40) and (371.35,33.43) .. (371.35,25.32) -- cycle ;
	\draw  [color={rgb, 255:red, 144; green, 19; blue, 254 }  ,draw opacity=1 ][line width=1]  (438.1,23.4) .. controls (438.1,16) and (448.32,10) .. (460.93,10) .. controls (473.55,10) and (483.77,16) .. (483.77,23.4) .. controls (483.77,30.8) and (473.55,36.79) .. (460.93,36.79) .. controls (448.32,36.79) and (438.1,30.8) .. (438.1,23.4) -- cycle ;
	\draw  [color={rgb, 255:red, 144; green, 19; blue, 254 }  ,draw opacity=1 ][line width=1]  (580.97,28.5) .. controls (580.97,22.16) and (588.83,17.02) .. (598.54,17.02) .. controls (608.24,17.02) and (616.1,22.16) .. (616.1,28.5) .. controls (616.1,34.84) and (608.24,39.98) .. (598.54,39.98) .. controls (588.83,39.98) and (580.97,34.84) .. (580.97,28.5) -- cycle ;
	\draw    (417.02,27.22) .. controls (460.35,324.5) and (458.01,117.81) .. (438.1,23.4) ;
	\draw    (354.95,158.64) -- (241.02,158.64) ;
	\draw [shift={(239.02,158.64)}, rotate = 360] [color={rgb, 255:red, 0; green, 0; blue, 0 }  ][line width=0.75]    (10.93,-3.29) .. controls (6.95,-1.4) and (3.31,-0.3) .. (0,0) .. controls (3.31,0.3) and (6.95,1.4) .. (10.93,3.29)   ;
	\draw  [color={rgb, 255:red, 144; green, 19; blue, 254 }  ,draw opacity=1 ][line width=1]  (477.16,94.6) .. controls (477.16,87.2) and (487.39,81.21) .. (500,81.21) .. controls (512.61,81.21) and (522.84,87.2) .. (522.84,94.6) .. controls (522.84,102) and (512.61,108) .. (500,108) .. controls (487.39,108) and (477.16,102) .. (477.16,94.6) -- cycle ;
	\draw  [draw opacity=0] (533.02,109.49) .. controls (527.23,108.39) and (522.81,102.58) .. (522.81,95.58) .. controls (522.81,95.25) and (522.82,94.92) .. (522.84,94.6) -- (535,95.58) -- cycle ; \draw   (533.02,109.49) .. controls (527.23,108.39) and (522.81,102.58) .. (522.81,95.58) .. controls (522.81,95.25) and (522.82,94.92) .. (522.84,94.6) ;  
	\draw  [draw opacity=0] (469.13,109.23) .. controls (473.72,107.94) and (477.18,102.24) .. (477.18,95.4) .. controls (477.18,95.14) and (477.17,94.87) .. (477.16,94.6) -- (467.15,95.4) -- cycle ; \draw   (469.13,109.23) .. controls (473.72,107.94) and (477.18,102.24) .. (477.18,95.4) .. controls (477.18,95.14) and (477.17,94.87) .. (477.16,94.6) ;  
	\draw [color={rgb, 255:red, 208; green, 2; blue, 27 }  ,draw opacity=1 ]   (394.18,40) .. controls (426,327) and (490,148) .. (494,106) ;
	\draw [color={rgb, 255:red, 208; green, 2; blue, 27 }  ,draw opacity=1 ]   (500,108) .. controls (496,157) and (530,152) .. (512,106) ;
	\draw [color={rgb, 255:red, 208; green, 2; blue, 27 }  ,draw opacity=1 ]   (87.95,13.19) .. controls (18,61) and (85,202) .. (119,198) ;
	\draw [color={rgb, 255:red, 208; green, 2; blue, 27 }  ,draw opacity=1 ]   (132,158) .. controls (178,276) and (146,23) .. (87.95,13.19) ;
	\draw [color={rgb, 255:red, 208; green, 2; blue, 27 }  ,draw opacity=1 ]   (132,158) .. controls (120,88) and (184,201) .. (119,198) ;
	\draw [color={rgb, 255:red, 208; green, 2; blue, 27 }  ,draw opacity=1 ]   (528,81) .. controls (591,137) and (528,159) .. (517,105) ;
	\draw [color={rgb, 255:red, 208; green, 2; blue, 27 }  ,draw opacity=1 ]   (474,36) .. controls (483,54) and (514,71) .. (528,81) ;

\end{tikzpicture}
\caption{A conic singular sub-manifold and its smooth model as a p-sub-manifold.}
\label{fig:globconic}
\end{figure}
\begin{proposition}\label{prop:csubisc}
A (globally) conic singular sub-manifold  equipped with the 
inner metric is a (globally) conic singular manifold.
\end{proposition}
\begin{proof}
	The restriction of a(n asymptotic) conic chart to a p-sub-manifold 
yields a(n asymptotic) conic chart onto the (globally) conic singular sub-manifold and the (globally) conic metric restricted to a p-sub-manifold is a (globally) conic metric on the p-sub-manifold. 
The (globally) conic singular sub-manifold with its inner metric 
is also a length space, where we can point out the length-minimizing curves using for instance the  quotient construction of Section~\ref{sec:quotientspace}.
\end{proof}
\begin{definition}\label{def:con-compl}
Let $X$ be a closed subset of the globally conic singular manifold $(E,d)$. The 
\em conic completion in $\overline{E}$ \em of the subset $X$ is simply its 
closure  in the conic completion $\overline{E}$.
\end{definition}
Note that the conic completion of a globally singular sub-manifold $X$ in $E$ is 
due to Proposition~\ref{prop:csubisc} isometric to a conic completion of $X$, 
whenever $X$ is considered with the inner metric. Therefore, conic completion 
gives a convenient  way to characterize a globally conic singular manifold
as follows.
\begin{proposition}\label{prop:acs-submanif-complet}
Let $(E,d)$ be a globally conic singular manifold and let $(\overline{E},\overline{d})$ be its conic completion.
The closed subset $X$  is a globally conic singular sub-manifold of $E$ if and only if 
its conic completion $\overline{X}$ in the conic completion $\overline{E}$ is a conic singular sub-manifold. 
\end{proposition}
%
%
%
%
%
%
%
%
%
%
%
%
%
%
%
%
%
%
%
%
%
%
%
%
%
%
%
%
%
%
%
%
%
%
\section{Lipschitz geometry on globally conic singular 
manifolds}\label{section:main}
By definition, a   manifold which is Riemannian,  conic singular 
or globally conic singular is LNE in itself as a length space. A 
smooth sub-manifold of a Riemannian manifold is always locally 
LNE, yet it may fail to be LNE, see \cite{CoGrMi3}. In this section we show that conic singular sub-manifolds are locally LNE and globally conic singular manifolds are LNE in an ambient globally conic singular manifold, generalizing results of \cite{CoGrMi3}.

\subsection{Preliminary remarks on LNE subsets of conic singular manifolds}
The next 
result generalizes \cite[Proposition 
2.1]{BiMo},\cite[Proposition 2.4]{KePeRu} and \cite[Lemma 
2.6]{CoGrMi1} in the context of conic singular manifolds.
%
%
\begin{proposition}\label{prop:compact-ic-LNE}
A connected compact subset of the conic singular manifold 
$(E,d)$ is LNE if and only if it is locally LNE. In particular, 
any connected compact $C^1$ embedded sub-manifold of $E^o$, 
possibly with boundary, is LNE.
\end{proposition}
\begin{proof}
The proof of \cite[Proposition 2.4]{KePeRu} is metric in nature, 
thus applies in this setting.
\end{proof}
The LNE property is hereditary in the 
sense of Proposition~\ref{propLNEHereditary} below. Equivalence 
of induced metrics on an LNE subset  shows 
that LNE subsets preserve ambient Lipschitz properties, for 
example H\"{o}lder exponents.
\begin{proposition}[Proposition 1.9, \cite{CoGrMi3}]\label{propLNEHereditary}
Let $X$ be a subset of a conic singular manifold $(E,d)$ which 
is locally LNE at a point $\bx$. A subset $Y$ of $X$ is locally 
LNE at $\bx$ in $(X,d_{\inn}^X)$ if and only if it is locally 
LNE at $\bx$ in $(E, d)$.
\end{proposition}
The next result is the following analogue of a cone type 
result in the Euclidean case.
\begin{proposition}\label{prop:cyl-con-LNE}
Let $(E,d)$ be a conic singular manifold and $\ba\in E_\sing$.
Let $\phi:\Cyl(N,\eta) \to U$ be a model conic chart on $E$
centred at $\ba$, let $g^c$ be the model conic metric on 
$\Cyl(N,\eta)$ and  $g_N$ be the Riemannian metric over
$N$ associated with $g^c$. If $Y$ is a LNE subset of $(N,g^\dd(0))$
then there exists a positive height $\eta_Y<\eta$ such that
each image $\phi(Y\times[0,\ve])$ is LNE in $(E,d)$ for any 
$0< \ve \leq \eta_Y$.  
\end{proposition}
\begin{proof}
The model conic metric $g^c$ writes as
$$
g^c := \rd r^2 + r^2 g^\dd(r)
$$
with $g_N = g^\dd(0)$. By Corollary~\ref{cor:conicequivalent}, up to shrinking 
$\eta$, it is equivalent as a singular Riemannian metric to 
the simple conic metric
$$
h^c = \rd r^2 + r^2 g_N 
$$
In particular the inner conic distances induced respectively by 
$g^c$ and $h^c$ over any subset of $\Cyl(N,\eta)$ are 
equivalent. 
	
Let $d^N$ be the distance on $N$ induced by $g_N$, and $d_{h^c}$ 
be the conic distance induced by $h^c$ on $Q =\Cyl(N,\eta)$. 
Let $X$ be a LNE subset of $N$. We recall Estimates 
\eqref{eq:conic-dist}
$$
\frac{|r-r'|}{2} + \min(r,r')\frac{d_N(\by,\by')}{2} \; \leq \;
d_{h^c} (\bx,\bx') \;\leq \; |r-r'| + \min(r,r')d_N(\by,\by').
$$
Let $d_N^Y$ be the inner distance on $Y$ obtained for $g_N$,
and let $d_{h^c}^Y$ be the inner conic distance over $Q$ 
obtained from $h^c$, that is the inner distance over $Q 
\setminus \dd Q$. Let $L$ be a LNE constant for $Y$.
For points $\bx,\bx' \in Q$ with $r'\geq r$, then clearly 
$$
d_{h^c}^Q(\bx,\bx') \geq |r-r'| \;\; {\rm and} \;\;  
d_{h^c}^Q(\bx,\bx') \geq r\,L\,d_N(\by,\by')
$$
Since the following estimates also holds true
$$
d_{h^c}^Q(\bx,\bx') \leq |r-r'| + r\,L \, d^N (\by,\by')
$$
we conclude that any cylinder $\Cyl(X,\ve)$ with $\ve \leq \eta$ 
is LNE w.r.t. to $h^c$. 
	
Let $d_{g^c}^Q$ be the inner (pseudo-)distance induced by $g^c$ 
over $Q$ and let $X := \phi(Q)$. By definition the spaces 
$(Q,d_{g^c}^Q)$ and $(X,d_\inn^X)$, are locally (pseudo-)isometric.
Since $g^c$ and $h^c$ are equivalent over $Q$, the result is proved.
\end{proof}
\subsection{Conic singular sub-manifolds are locally Lipschitz 
normally embedded}
\begin{theorem}\label{thm:main-compact}
Let $X$ be a subset of the conic singular manifold $(E,d)$ such 
that the germ $(X,\ba)$ at the point $\ba\in E$ is that of a 
closed conic singular sub-manifold. Then $X$ is locally LNE at 
$\ba$.
\end{theorem} 
\begin{proof}
Without loss of generality we can assume that $X\setminus \ba$ 
is a smooth  sub-manifold of a smooth manifold $E\setminus\ba$. 
There exists a conic chart $\phi :P =\Cyl(N,\eta) \to U$, 
centred at $\ba$ in $E$ such that $\phi^*g^o$ extends smoothly 
as the model conic metric $g^c$ over $P$. 
	
By definition, the closure $\cX$ of $\phi^{-1}(X\setminus \ba)$ 
in $P$ is a p-sub-manifold. 
It suffices to show that the outer conic distance and inner conic distance 
over $\cX$ obtained from the conic metric $g^c$ are 
equivalent.
	
Since $(\by,r)$ are global coordinates over $P$, the model
conic metric $g^c$ writes
$$
g^c = \rd r^2 + r^2 g^\dd(r).
$$
Since the  LNE property by Corollary \ref{cor:conicequivalent} stays true regardless of the conic 
metric we are working with, we can assume without loss of 
generality, up to shrinking $\eta\leq 1$, that $g^c$ is a simple conic metric, i.e. the family $r\to g^\dd(r)$ is constant  and equal 
to the associated Riemmanian metric on the boundary $g_N :=g^\dd(0)$ and there exists a collar 
neighbourhood diffeomorphism 
$$
\tau : \cX \to \cC = \dd\cX \times [0,\eta), \;\; (\by,r) 
\mapsto (\omg(\by,r),r)). 
$$
\begin{lemma}\label{rmk:bi-lip-cone}
The mapping 
$\phi\circ\tau\circ\phi^{-1} : (\phi(\cX),d) \to (\phi(\cC),d)$ 
is bi-Lipschitz for the outer metric space structures. 	
\end{lemma}
\begin{proof}
Consider $P$ equipped with the product metric 
$$
\cyl := \rd r^2 + g_N.
$$ 
Let $|-|_\cyl$ be the induced norm on $TP$ and $d^\cyl$ be the 
induced distance on $P$. Let $|-|_N$ be the norm on $TN$ 
obtained from $g_N$, and $d^N$ be the induced distance on $N$. 
Last, let $|-|_c$ be the semi-norm induced by $g^c$ on $TP$ and 
$d^c$ be the conic distance on $P$. We recall 
that, given $\bv = (\xi,\lbd) \in T_{(\by,r)} P = T_\by N \times 
\R$, the norms of $\bv$ are as follows 
$$
|\bv|_\cyl^2 = \lbd^2 + |\xi|_N^2 \;\; {\rm and} \;\; 
|\bv|_c^2 = \lbd^2 + r^2 \,|\xi|_N^2.
$$
Observe that the identity mapping $(P,d^\cyl) \to (P,d^c)$ is 
(pseudo-)Lipschitz with Lipschitz constant $1$ whenever $\eta 
\leq 1$.  
	
Observe that the restriction of $|-|_\cyl$ to $TN$ is simply 
$|-|_N$. Given $\bz = (\by,r) \in \cX$, the (restriction to 
$T\cX$ of the) $\cyl$-norms for the deriatives of  smooth mappings  $\tau$ and 
$\omg$ are 
$$
\|D_\bz \tau\|_\cyl \;\; {\rm and} \;\; \|D_\bz \omg\|_\cyl, 
$$
Analogously, over $\cX\setminus \dd\cX$, the (restriction to $T\cX$ 
of the) $g^c$-norms for $\tau$ and $\omg$ is denoted by 
$$
\|D_\bz \tau\|_c\;\; {\rm and} \;\; \|D_\bz \omg\|_c, 
$$
Given a subset $S$ of $P$ and $r\geq 0$, let 
$$
S_r : = S\cap (N\times r).  
$$
Since $\omg|_{\dd \cX\times 0}$ is a diffeomorphism and 
$\tau(\cX_r)= \cC_r$, the mapping $\dd_N \omg(\bz) = 
D_\bz \omg|_{T_\bz \cX_r}$ is invertible, up to shrinking $\eta$ 
there exists a positive constant $K$ such that 
$$
\frac{1}{K}\leq \|\dd_N \omg(\bz)\|_\cyl \leq K, \;\;\forall \; 
\bz\in \cX.
$$
Observe that
$$
\|\dd_N \omg(\bz)\|_\cyl = \|\dd_N\omg(\bz)\|_c.
$$
Let $\bv = (\xi,\lbd) \in T_\bz \cC = T_\by\dd\cX\times\R 
\subset	T_\bz P = T_\by N \times \R$ with $\bz\in 
\cX$. We get
$$
D_\bz\omg\cdot \bv = \dd_N\omg(\bz)\cdot\bv + \lbd \dd_r\omg.  
$$
Thus there exists a positive constant $A$ so that 
$$
|(\dd_r \omg)(\bz)|_\cyl \leq A \cdot \|\dd_N\omg(\bz)\|_\cyl, 
\;\;\; \forall \; \bz\in \cX, 
$$
which yields
$$
|D_\bz\omg\cdot \bv|_c = r \cdot |D_\bz\omg\cdot \bv|_\cyl \leq 
r B \cdot \|\dd_N\omg(\bz)\|_\cyl\cdot|\bv|_\cyl, \;\; {\rm 
where} \;\;\ B = \sqrt{1+A^2}.
$$
Thus
\begin{equation}\label{eq:norm-c-mu}
\|D_\bz\omg\cdot \bv\|_c \leq B \cdot \|\dd_N\omg(\bz)\|_\cyl, 
\;\; \bz\in\cX.
\end{equation}
Let $\bun_\R$ be the identity mapping of $\R$, the latter 
being identified with $T_\bz (0,r_0)$ and $T_{\tau(\bz)} 
(0,r_0)$. Since $D_\bz \tau = D_\bz\omg \oplus \bun_\R$, we 
deduce that 
$$
\|D_\bz \tau\|_c \leq 1+ r\cdot\|D_\bz \omg\| \leq 1 + B\cdot 
|\dd_N\omg(\bz)\|_\cyl
$$
Thus $\tau$ induces (pseudo-)Lipschitz homeomorphisms 
$(\cX_i,d^c) \to (\cC_i,d^c)$, where $\cX_1,\ldots,\cX_s,$ are 
the connected components of $\cX$ and $\cC_i = \tau(\cX_i)$. 
Since the smooth inverse $\tau^{-1}$ is of the form $(r,\by)
\to (\mu(r,\by),r)$ the exact same arguments show
that $\tau^{-1}$ is also (pseudo-)Lipschitz with respect to the outer 
conic metrics $d^c$ over each $\cC_i$. 
	
Since $\dd\cX = \cup_{i=1}^s \cS_i$, where $\cS_i$ is a smooth 
compact connected sub-manifold of $\bS^{n-1}$ if $\ba\in 
E_\sm^d$ or of $N$ if $\ba \in E_\sing^d$, and $\cS_i \cap 
\cS_j = \emptyset$ for $1\leq i < j \leq s$, we deduce that
$$
\dlt := \min\{d^\cyl(\cS_i,\cS_j) : 1\leq i < j\leq k\} >0.
$$ 
From which follows that for any $r\leq \eta$, up to shrinking 
$\eta$, we get 
$$
\min\{d^c(\cX_i)_r,(\cX_j)_r) : 1\leq i < j\leq k\} 
\geq r \frac{\dlt}{2}.
$$
Fix a pair $1\leq i < j \leq s$ and let $\bz_k = (\by_k,r_k) 
\in \cX_k$ with $k=i,j$. Then  
$$
d^c(\bz_i,\dd P) + d^c(\dd P,\bz_j) \geq d^c(\bz_i,\bz_j) \geq
d^c ( (\cX_i)_{r_i},(\cX_j)_{r_j} ).
$$
Since $\tau|_{\dd\cX}$ is the identity mapping of $\dd\cX$, up 
to shrinking $\eta$, we deduce the following estimates
$$
r_i + r_j \geq d^c(\bz_i,\bz_j) \geq \frac{\dlt}{4} \, [r_i + 
r_j]
$$
When $\bw_k = \tau(\bz_k) = (\omg_k,r_k)$, for $k=i,j$, we get
$$
d^c(\bw_i,\bw_j) \leq  d^c(\bw_i,\dd M) + d^c(\dd M,\bw_j) =
r_i + r_j \leq \frac{4}{\dlt} \, d^c(\bz_i,\bz_j)  
$$
We conclude that $\tau: (\cX\setminus\dd\cX,d^c) \to 
(\cC\setminus\dd\cC,d^c)$ is (pseudo-)Lipschitz. The same 
arguments will give that $\tau^{-1}$ is (pseudo-)Lipschitz as 
well. 
\end{proof}
Since the image of the cylinder $\phi(\cC)$ is LNE with respect to the 
metric $d$ by Proposition \ref{prop:cyl-con-LNE}, so is 
$\phi(\cX)$. Therefore $X$ is locally LNE at $\ba$.
\end{proof}
By Theorem \ref{thm:main-compact}, Lemma \ref{rmk:bi-lip-cone} 
and Proposition \ref{prop:cyl-con-LNE} we get a family of LNE 
representatives as follows.
\begin{corollary}\label{cor:representatives}
If $X$ is a germ at a point $\ba $ of a conic singular 
sub-manifold in a conic singular manifold $(E,d)$, then there 
exists a positive number $\eta_X$ and a germ of a  non-negative 
function $r:(E,\ba)\to \R$, smooth outside of $\ba$ and 
vanishing only at the point $\ba$ such that $X\cap 
r^{-1}([0,\eta])$ is locally LNE at $\ba$ for every 
$\eta<\eta_X$. 
\end{corollary} 
Theorem \ref{thm:main-compact} and Proposition 
\ref{prop:compact-ic-LNE} yield the following relevant 
implications.
\begin{corollary}
A  conic singular sub-manifold of a conic singular manifold
is locally LNE. 
\end{corollary}
\begin{corollary}\label{cor:compactLNE}
A connected compact conic singular sub-manifold of a conic singular 
manifold is LNE. 
\end{corollary}
\subsection{Globally conic singular sub-manifolds are Lipschitz 
normally  embedded}
\begin{theorem}\label{thm:MainAsympConic}
Let $X$ be a globally conic singular sub-manifold of the
globally conic singular manifold $(E,d)$. Then each
connected component of $X$ is LNE in $E$.
\end{theorem}

\begin{proof}
Without loss of generality we can assume that $X$ is connected.

Let $(\Ebar,\dbar)$ be a conic completion of $(E,d)$ and let $Y:=\ovX$ be the 
closure of $X$ in $\Ebar$. Let $\infty_{X}=\infty_E\cap Y$ be the conic ends  of 
$X$. Since $(Y,\dbar_\inn^{Y})$ is a conic completion of the
globally conic singular manifold $(X,d_\inn^X)$,  we get
$$
\dbar_\inn^{Y}|_X = \dbar_\inn^X.
$$
Since $Y$ is a conic singular sub-manifold of the compact conic singular manifold 
$\Ebar$, by Corollary~\ref{cor:compactLNE} it is LNE, thus there exits a positive 
constant $L'$ such that 
$$
\dbar_{Y}\leq \dbar_\inn^{Y} \leq L' \, \dbar_{Y}.   
$$
Corollary \ref{cor:conic-complet}, restricting from $\Ebar$ to $Y$, implies there 
exist positive constants $A', B'$ such that
$$
A'\,\dbar_{Y} (\bx,\bx')\;\leq\; r\,r'\,d_X(\bx,\bx') \;\leq \; B'\,\dbar_{Y}(\bx,\bx').
$$
Define  $\tr = \dbar_\inn^{Y}(\bx,\infty_{X})$ and $\tr' = \dbar_\inn^{Y}(\bx',\infty_{X})$,
Corollary \ref{cor:conic-complet} applied to $Y$ with the inner distance $\dbar_\inn^{Y}$ gives positive constants $A'', B''$ such that
$$
A''\,\dbar_\inn^{Y} (\bx,\bx')\;\leq\; \tr\,\tr'\,d_\inn^X(\bx,\bx') \;\leq \; B''\,\dbar_\inn^{Y}(\bx,\bx'). 
$$
Therefore  over $X$  we get
$$
d_X \leq d_\inn^X \leq L\,d_X \;\; {\rm with} \;\; L = \frac{B''}{A'}\,(L')^3,
$$
since $\dbar_{Y}\leq \dbar_\inn^{Y}\leq L' \,\dbar_{Y}$, yielding $r\leq \tr \leq L' \, r$.
\end{proof}

By Theorems~\ref{thm:main-compact} and~\ref{thm:MainAsympConic}  we easily get the following variation, of interest for applications. 
\begin{corollary}\label{thm:LNE-LNE-scattering}
Let $(E,d)$ be a globally conic singular manifold and let $X$ be its  closed 
connected subset  which is asymptotically conic. Then $X$ is LNE in $(E,d)$ if
and only if its conic completion $\overline{X}$ is LNE in $(\Ebar,\dbar)$. 
\end{corollary}
\subsection{Examples}\label{section:ex}
The Euclidean space $(\Rn,|-|)$ is naturally a globally conic 
singular manifold and its conic completion is simply $\bS^n$ via 
the inverse of the stereographic projection.
The main result of \cite{CoGrMi2} states that each connected 
component of  a generic real algebraic set of $\Rn$ is a 
globally conic-singular sub-manifold 
(usually non-singular) of $\Rn$, thus is LNE.

Consider the following singular Riemannian metric over $\R^2$:
$$
h = [2x^2-2xy+y^2]\rd x^2 + 2[x^2+xy-y^2]\rd x\rd y + 
[x^2+2xy+2y^2]\rd y^2.
$$
Denoting $r = \sqrt{x^2+y^2}$, we check that $h$ is equivalent to $r^2eucl$ over 
$\R^2$, where $eucl$ is the Euclidean metric tensor. Therefore $g^o = r^{-2}h$
is a Riemannian metric over $E^o = \R^2\setminus \bbo$. Let $E = 
\R^2$ and $d$ be the length distance over $E$ induced from 
$g^o$. Since $g^o$ is equivalent to $eucl|_{E^o}$, the 
space$(E,d)$ is a complete length space. 
In polar coordinates, we find
$$
g^o = 2\rd r^2 + 2r\rd r \rd \tht + r^2 \rd \tht^2.
$$
The following mapping is continuous and surjective
$$
\phi : M = \bS^1\times\Rgo\to E, \;\; (\tht,r)\mapsto 
\left\{
\begin{array}{ccc}
{r\,e^{i(\tht-\ln r)}} & {\rm if} & r>0 \\
0 & {\rm if} & r=0	
\end{array}
\right. ,
$$
and over $M\setminus\dd M$, we verify that 
$$
\phi^*(g^o) = \rd r^2 + r^2\rd \tht^2,
$$
thus extends as a conic 
metric over $M$. Thus $(E,d)$ is a conic singular manifold. 
Using the conic inversion and $\phi$ we check that $(E,d)$ is 
also asymptotically conic.

Observe that although $\phi$ is not differentiable over $\bS^1\times 0$ as 
a mapping with values in $\R^2$, it is yet a globally conic 
resolution of $(E,d)$ in the sense of Proposition~\ref{propGLOBALchart}.

Observe also that the (image of the unit speed) $g^o$-geodesic 
from $\bx_0 = r_0\,e^{\tht_0}$ to $\bbo$ is the 
logarithmic spiral parametrized in polar coordinates as $r \to 
(\tht_0+\ln r-\ln r_0,r)$.
Obviously, a similar phenomenon holds true for any 
$g^o$-geodesic leaving any compact of $E$.
%
%
%
%
%
%
%
%
%
%
%
%
%
%
%
%
%
%
%
%
%
%
%
%
%
%
%
%
%
%
%
%
%
%
%
%
%
%
%
%
%
%
%
%
%
%
%
%
%
%
%
%
%
%
%
%
%
%
%
%
%

\bibliographystyle{plain}
\bibliography{conicbiblio}
\end{document}